\documentclass[12pt]{amsart}
\usepackage{amssymb,amsthm,amsmath,a4}

\newtheorem{theorem}[equation]{Theorem}

\newtheorem{lemma}[equation]{Lemma}

\newtheorem{corollary}[equation]{Corollary}

\newtheorem{conj}[equation]{Conjecture}
\newtheorem{problem}[equation]{Problem}

\theoremstyle{remark}
\newtheorem{remark}[equation]{Remark}
\theoremstyle{definition}

\newtheorem{cond}[equation]{Condition}
\newtheorem{condition}[equation]{Condition}
\newtheorem{example}[equation]{Example}

\numberwithin{equation}{subsection}

\renewcommand{\qed}{\hspace*{\fill} \setlength{\unitlength}{1mm}
\begin{picture}(2.5,2.5)
      \put(0,0){\framebox(2.5,2.5){}}
  \end{picture}
\setlength{\unitlength}{1pt}}

\newcommand{\supp}{\operatorname{supp}}
\newcommand{\re}{\operatorname{Re}}
\newcommand{\im}{\operatorname{Im}}

\newcommand{\diam}{\operatorname{diam}}

\newcommand{\sign}{\operatorname{sign}}

\newcommand{\Tr}{\operatorname{Tr}}

\newcommand{\eps}{\varepsilon}
\newcommand{\Vol}{\operatorname{Vol}}

\newcommand{\de}{{\delta}}
\newcommand{\vk}{{\varkappa}}
\newcommand{\la}{{\lambda}}
\newcommand{\si}{{\sigma}}
\newcommand{\vf}{\varphi}
\newcommand{\C}{\mathbb{C}}
\newcommand{\Z}{\mathbb{Z}}
\newcommand{\R}{\mathbb{R}}
\newcommand{\Sbb}{\mathbb{S}}

\newcommand{\MC}{\mathcal{M}}

\newcommand{\dr}{\mathrm{d}}

\def\leq {\leqslant}
\def\geq {\geqslant}

\begin{document}
\title[Average growth of the spectral function]
{Average growth of the spectral function\\ on a Riemannian manifold}

\title[Average growth of the spectral function]
{Average growth of the spectral function\\ on a Riemannian manifold}
\author[H. Lapointe]{Hugues Lapointe}
\address{D\'e\-par\-te\-ment de math\'ematiques et de
sta\-tistique, Univer\-sit\'e de Mont\-r\'eal, CP 6128 succ
Centre-Ville, Mont\-r\'eal, QC  H3C 3J7, Canada.}
\email{lapointe@dms.umontreal.ca}

\author[I. Polterovich]{Iosif Polterovich}
\address{D\'e\-par\-te\-ment de math\'ematiques et de
sta\-tistique, Univer\-sit\'e de Mont\-r\'eal, CP 6128 succ
Centre-Ville, Mont\-r\'eal, QC  H3C 3J7, Canada.}
\email{iossif@dms.umontreal.ca}

\author[Y. Safarov]{Yuri Safarov}
\address{Department of Mathematics, King's College,  Strand, London WC2R 2LS, UK}
\email{yuri.safarov@kcl.ac.uk}

\begin{abstract}
We study average growth of the spectral function of the Laplacian on
a Riemannian manifold. Two types of averaging are considered: with
respect to the spectral parameter and  with respect to a point on a
manifold. We obtain as well related estimates of the growth of the
pointwise $\zeta$-function along vertical lines in the complex
plane. Some examples and open problems regarding almost periodic
properties of the spectral function are also discussed.
\end{abstract}


\maketitle

\section{Introduction}\label{intro}
\subsection{Notation}
Throughout the paper, we use the following notation.
\begin{enumerate}
\item[$\bullet$]
$C$ denotes various positive constants whose precise values are not
important for our purposes; dependence on various parameters is
indicated by lower indices.
\item[$\bullet$]
$d_{x,y}$ is the Riemannian distance between $x$ and $y$.
\item[$\bullet$]
$\R_+:=(0,+\infty)$.
\item[$\bullet$]
$t,s,\la,\mu,\tau\in\R$ and $\xi\in\R^d$.
\item[$\bullet$] $\la_+:=\max\{\la,0\}$,
    $\lceil\la\rceil=\min\{n\in \mathbb{Z}|\, \la \leq n\}$ is the
    ceiling function and $\lfloor \la \rfloor=\max\{n\in \mathbb{Z}|\, \la \geq n\}$ is the
    floor function.
\item[$\bullet$]
$\hat f(\xi)=\int f(x) e^{-ix\xi}\, \dr x$ is the Fourier transform
of $f$ and $(f)^\vee$ is the inverse Fourier transform, so that
$(\hat f)^\vee=f$.
\item[$\bullet$]
$\,*\,$ denotes the convolution on $\R$.
\item[$\bullet$]
$\Gamma$ is the gamma-function and $\,\binom mk\,$ are the binomial
coefficients.
\item[$\bullet$]
$J_\alpha$ is the Bessel function of the first kind of order
$\alpha$.
\end{enumerate}

Let $f$ and $g$ be real-valued functions on $\R_+$, and let $g>0$.
We write
\begin{enumerate}
\item[$\bullet$]
$f(\la)=O(\la^p)\,$ if
$\,\limsup_{\la\to+\infty}\la^{-p}\,f(\la)<\infty$;
\item[$\bullet$]
$f(\la)=o(\la^p)\,$ if $\,\lim_{\la\to+\infty}\la^{-p}\,f(\la)=0$;
\item[$\bullet$]
$f(\la)=O(\la^{-\infty})\,$ if $\,\la^{m}f(\la)\to0\,$ as
$\,\la\to+\infty\,$ for all $\,m\in\R_+\,$;
\item[$\bullet$]
$f(\la)\ne o(g(\la))$ if there exists a sequence $\mu_n\to+\infty$
such that $\frac{|f(\mu_n)|}{g(\mu_n)}\geq C>0$ for all $n$;
\item[$\bullet$]
$f(\la)\gg g(\la)$ if there exists $\la_0$ such that
$f(\la)>C\,g(\la)$ for all $\la>\la_0$.
\end{enumerate}

\subsection{Spectral function}
Let $M$ be a compact Riemannian $n$-dimensional manifold and
$\Delta$ be the Laplacian on $M$ with the eigenvalues
$0=\la_0<\la_1^2\leq \la_2^2 \leq  \dots \la_j^2 \leq \dots$ and the
corresponding orthonormal basis of eigenfunctions $\{\phi_j(x)\}$.
Let $N_{x,y}(\la)$ be the {\it spectral function},
\begin{equation}
\label{specfundef} N_{x,y}(\la)\ :=\
\sum_{0<\la_j<\la}\phi_j(x)\,\overline{\phi_j(y)}\,,
\end{equation}
The spectral function $N_{x,y}(\la)$ is the integral kernel of the
spectral projection of the operator $\,\sqrt\Delta\,$, corresponding
to the interval $\,(0,\la)\,$. This way the spectral function could
be defined in the non-compact case as well, see Example \ref{Rn}.

By the Weyl formulae,
\begin{equation}\label{intro:bound1}
N_{x,x}(\la)\ =\ \frac{\la^{n}}{(4\pi)^{\frac{n}{2}}
\Gamma(\frac{n}{2}+1)}\ +\ O(\la^{n-1})
\end{equation}
and
\begin{equation}\label{intro:bound2}
N_{x,y}(\la)\ =\ O(\la^{n-1})
\end{equation}
for all $x \ne y\in M$ (see \cite{Ava, Lev, Ho1}. The estimates
\eqref{intro:bound1} and \eqref{intro:bound2} are sharp and attained
on a round sphere. However, the bounds can be improved to
$o(\la^{n-1})$ instead of $O(\la^{n-1})$, provided that the set of
geodesic loops originating at $x$ and the set of geodesics joining
$x$ and $y$ are of measure zero (see \cite{S1} or \cite{SV}).

The following condition  is used to prove the lower bounds on the
spectral function (see subsection \ref{proof-lower:theorem}).
\begin{cond}\label{intro:non-conj}
The points $x$ and $y$ are not conjugate along any shortest geodesic
segment joining them.
\end{cond}
It follows immediately from \cite[Corollary 18.2]{Mil} that for each
$x \in M$,  Condition \ref{intro:non-conj} is true for  almost all
$y \in M$. As was shown in \cite[Theorem 1.1.3]{JP} that Condition
\ref{intro:non-conj} implies
\begin{equation}\label{intro:bound3}
N_{x,y}(\la)\ \ne\ o\left(\la^{\frac{n-1}{2}}\right)\,.
\end{equation}

\subsection{Discussion}
\label{intro:disc} Let $\si_{x,y}(t)$ be defined by
$(\si_{x,y}(t))^\vee=\sum_j\phi_j(x)\,\phi_j(y)\,\de(\la-\la_j)\,.$
By the spectral theorem, $\si_{x,y}(t)$ coincides with the
distributional kernel of the operator $\exp(-it\sqrt{\Delta})\,$ or,
in other words, with the fundamental solution of the corresponding
hyperbolic equation. Proofs of
\eqref{intro:bound1}--\eqref{intro:bound3} are based on the study of
singularities of the distribution $\si_{x,y}(t)$, which is usually
done by means of Fourier integral operators (see, for example,
\cite{DG} or \cite{SV}). It is well known that these singularities
lie in the set ${\mathcal L}_{x,y}$ of lengths of all geodesic
segments $x,y \in M$. In particular, $\si_{x,y}(t)$ with $y\ne x$ is
infinitely smooth in a neighbourhood of the origin, whereas
$\si_{x,x}(t)$ has a strong singularity at $t=0$ related to the main
term in the Weyl formula \eqref{intro:bound1}. If
Condition~\ref{intro:non-conj} is fulfilled then the singularity at
$\,t=d_{x,y}\,$ can be explicitly described, which leads to the
lower bound \eqref{intro:bound3}.

The purpose of this paper is to study the {\it average} growth of
$N_{x,y}(\la)$, where averaging is considered either with respect to
the spectral parameter $\la$ or with respect to a point $y$. The
following examples indicate that, while both estimates
\eqref{intro:bound2} and \eqref{intro:bound3} are sharp,  the
typical order of growth of the spectral function is more likely to
be $\la^{\frac{n-1}{2}}$ rather than $\la^{n-1}$.

\begin{example}\label{Rn}(cf. \cite[Example 4.1]{Peetre})
For the Laplacian in $\R^n$ the spectral function is given by
$$
N_{x,y}(\la)\ =\ \frac{1}{(2\pi)^n} \int_{|\xi| \leq \la}
e^{i\langle x-y,\xi\rangle}\,\dr\xi\,
=\frac{\la^n}{(2\pi)^n}\,\hat{\chi}_{B^n}(\la\,(x-y))\,
$$
where $\,\chi_{B^n}\,$ is the characteristic function of the unit
ball. Taking into account \cite[formula (4.4)]{Peetre} and
\cite[formula 3.771(8)]{GR},  one gets
$\,N_{x,y}(\la)=\mathcal{N}(\la,d_{x,y})\,$ where
\begin{equation}
\label{specrn} \mathcal{N}(\la,d_{x,y})\ :=\ (2\pi)^{-\frac{n}{2}}
\,d_{x,y}^{\,-\frac{n}{2}}\, \la^{\frac{n}{2}}\,
J_{\frac{n}2}\left(d_{x,y}\,\la\right)
\end{equation}
and $\,d_{x,y}=|x-y|\,$. Using \cite[formula 8.451(1)]{GR} we obtain
for all $x\ne y$:
\begin{equation}
\label{asympchar}  \mathcal{N}(\la,d_{x,y})\ =\ \frac{2\,
\la^{\frac{n-1}{2}}}{(2\pi
d_{x,y})^{\frac{n+1}{2}}}\,\sin\left(\la\,d_{x,y}-
\frac{(n-1)\pi}{4}\right) +
O\left(\frac{\la^{\,(n-3)/2}}{d_{x,y}^{\,\,(n+3)/2}}\right).
\end{equation}
\end{example}

\begin{example}\label{intro:spheres}
Let $\Sbb^n$ be a round sphere of dimension $n\geq 2$. Given
$x\in \Sbb^n$, denote by  $-x$ the diametrically opposite point.
Then $N_{x,y}(\la)=O\left(\la^{\frac{n-1}{2}}\right)$ for all $y
\neq \pm x$ and $N_{x,-x}(\la)\ne o(\la^{n-1})$.
\end{example}

\subsection{Plan of the paper}
The paper is organized as follows. In the next section we present
our main results. In Section \ref{sect-conj} we discuss some
examples and open problems regarding almost periodic properties of
the spectral function. In Section \ref{proof-upper} we give proofs
of the upper bounds formulated in the subsections 2.1--2.4. Proofs
of the lower bounds stated in the subsection 2.5 are given in
Section \ref{proof-lower}. All examples are justified in
Section~\ref{examples}. Finally, in Section \ref{zeta} we establish
the estimates on the $\zeta$-function presented in the subsection
2.6.


\section{Main results}\label{main}
\subsection{Average over the manifold}
\label{main:upper} The principal object of study in the present
paper is the {\it rescaled spectral function}
\begin{equation}\label{intro:specfrenorm}
\tilde N_{x,y}(\la)\ :=\ \la^{\frac{1-n}{2}}\,N_{x,y}(\la)
\end{equation}
on a compact $n$-dimensional Riemannian manifold $M$. Note that, by
\eqref{asympchar}, in the Euclidean case the rescaled spectral
function $\,\tilde {\mathcal{N}}(\la,d_{x,y}):=\la^{\frac{1-n}{2}}\,
\mathcal{N}(\la,d_{x,y})\,$ is bounded for each fixed $x\ne y\,$.
Our main result is

\begin{theorem}\label{upper:theorem}
For any compact $n$-dimensional Riemannian manifold $M$ there exists
a constant $\,C_M\,$ such that
\begin{equation}\label{upper:int0}
\int_M\left|\tilde N_{x,y}(\la)- \tilde
{\mathcal{N}}(\la,d_{x,y})\right|^2\,\dr y\ \leq\
C_M\,,\qquad\forall\, x\in M\,,\ \forall\,\la\in\R_+\,,
\end{equation}
where $d_{x,y}$ is the Riemannian distance between $x,y \in M$.
\end{theorem}

\begin{remark} Note that $\tilde N_{x,x}(\la)\gg
\la^{\frac{n+1}{2}}$. We subtract the term $\tilde
{\mathcal{N}}(\la,d_{x,y})$ in \eqref{upper:int0} in order to
``regularize'' the rescaled spectral function in a neighbourhood of
the diagonal. Geometrically this can be interpreted as follows.
Consider a normal coordinate system centred at $x$. In these
coordinates the Riemannian metric is Euclidean at the point $x$, and
hence near $x$ the geometric Laplacian $\Delta$ can be viewed as a
{\it perturbation} of the Euclidean Laplacian (a similar idea was
used in \cite[Section 3]{Pol} and goes back to \cite[Theorem
6.1]{AK}). We justify the regularization using the Hadamard
parametrix for the wave kernel, see the
subsection~\ref{proof-upper:n0}.
\end{remark}
\subsection{Spectral average}
Theorem~\ref{upper:theorem} implies a number of
results. The first one describes the  growth of the spectral
function on average with respect to $\la$.

\begin{theorem}\label{upper:cor1}
For every finite measure $\,\nu\,$ on $\R_+$ and each fixed $\,x\in
M\,$, there exists a subset $\,M_{x,\nu}\subset M\,$ of full measure
such that
\begin{equation}\label{upper:int2}
\int_0^\infty|\tilde N_{x,y}(\la)|^2\,\dr\nu(\la)\ <\
\infty\,,\qquad\forall y\in M_{x,\nu}\,.
\end{equation}
\end{theorem}
It should be emphasized that, generally speaking, the set
$\,M_{x,\nu}\,$ depends on the choice of $\,\nu\,$. Otherwise
\eqref{upper:int2} would mean that $\,\tilde N_{x,y}(\la)\,$ is
bounded for almost all $\,y\in M\,$, which is not always the case
(see the next subsection).

In particular, taking
$\,\dr\nu=(\la+1)^{-1}\,(\ln\la)^{-1-\eps}\,\dr\la\,$ and applying
Theorem~\ref{upper:cor1}, we see that
\begin{equation}\label{upper:log}
\int_0^\infty\la^{-1}\,(\ln\la)^{-1-\eps}\,|\tilde
N_{x,y}(\la)|^2\,\dr\la\ <\ \infty
\end{equation}
for all $\eps>0$, all $x\in M$ and almost all $y\in M$.

\begin{remark}\label{upper:randol}
This paper was inspired by \cite{Randol}, where the estimate
\eqref{upper:int2} was proved for surfaces of constant negative
curvature under the assumption that
$\dr\nu(\la)=(\la+1)^{-1-\eps}\,\dr\la$ with some $\eps>0$. Randol's
proof is based on the estimate \eqref{zeta:l2} for the pointwise
$\zeta$-function of the Laplacian and uses some results of complex
analysis obtained in \cite{HIP}. We use a more direct approach that
is applicable in higher generality and gives better estimates. It
also allows us to improve the bounds for the $\zeta$-function. Even
though these estimates are not needed in our proof, they seem to be
of independent interest. We have included them in
Section~\ref{main:zeta}, where we also explain the relation between
Theorem~\ref{upper:cor1} and properties of the $\zeta$-function.
\end{remark}
\subsection{Growth of the rescaled spectral function}
The\-o\-rem~\ref{upper:cor1} and Example~\ref{intro:spheres} suggest
that $\,\tilde N_{x,y}(\la)=O(1)\,$ for almost all $x,y\in M$ on any
manifold $M$. However, this is not true. In particular, for any
negatively curved manifold $M$ there exists a constant $\alpha_M >
0$ depending on certain dynamical properties of the geodesic flow,
such that $\tilde N_{x,y}(\la)\neq o\left((\ln
\la)^{\alpha_M}\right)$ for all $x,y\in M$ (see \cite{JP}). At the
same time, for fixed $x$ and $y$ on a negatively curved manifold,
the sequence of $\la$-s yielding the logarithmic growth of the
rescaled spectral function is very scarce. This sequence is quite
sensitive to the choice of the points $x$ and $y$. In particular,
for fixed $x$ and $\la$, the set of points $y$, for which the
function $\tilde N_{x,y}(\la)$ has a logarithmic ``peak'', is
expected to be very small.

The theorem below shows that a similar effect takes place on any
manifold, not necessarily negatively curved.

\begin{theorem}\label{upper:cor3}
There exists a constant $\,C_0\,$ not depending on $M$, and  a
constant $C_M$ depending on the geometry of $M$, such that for any
$C> C_0$, the measure of the set
\begin{equation}\label{upper:omega}
\Omega_x(\la,\mu)\ :=\ \{\,y\in M\,:\,|\tilde N_{x,y}(\la)|\geq
C\,(\mu+d_{x,y}^{\,-\frac{n+1}2})\,\}
\end{equation}
satisfies
$${\rm meas}(\Omega_x(\la,\mu))\leq\frac{C_M}{C^2 \mu^2}$$
for any point $x\in M$ and any $\la \in \R_+$.
\end{theorem}

\begin{corollary}\label{upper:cor4}
Let $\{\mu_k\}$ and $\{\tau_k\}$ be positive increasing sequences
converging to $+\infty$. Then, for each fixed $x\in M$, the measure
of the set of points $y\in M$ such that
\begin{equation}\label{upper:seq1}
|\tilde N_{x,y}(\tau_k)|\ \ge\ \mu_k\,,\qquad\forall k=1,2,\ldots,
\end{equation}
is equal to zero. Moreover, if $\,\sum_k\mu_k^{-2}<\infty\,$ then
\begin{equation}\label{upper:seq2}
\limsup_{k\to\infty}|\mu_k^{-1}\tilde N_{x,y}(\tau_k)|\ =\ 0
\end{equation}
for almost all $y\in M$.
\end{corollary}

\subsection{Weighted average over the manifold}
Let us consider  the average of the rescaled spectral function over
$M$ with a weight given by a power of the distance function
$d_{x,y}$. 
\begin{theorem}\label{upper:cor0}
For all $\vk\geq0$, $\,x\in M\,$ and $\la \in \R_+$, we have
\begin{equation}\label{upper:int1}
\int_M d_{x,y}^{\,\vk}\,|\tilde N_{x,y}(\la)|^2\,\dr y\ \leq\
\begin{cases}
C_{\vk,\,M}\, (1+\la^{1-\vk}), \,&\vk\ne1\,,\\
C_M\,(1+ |\ln\la|)\,,&\vk=1.
\end{cases}
\end{equation}
\end{theorem}

Note that $\,\int_M|\tilde N_{x,y}(\la)|^2\,\dr
y=\la^{1-n}\,N_{x,x}(\la)=C \la +O(1)\,$. Therefore the estimate
\eqref{upper:int1} with $\vk=0$ is order sharp. We also remark that,
as follows from the proof of Theorem \ref{upper:cor0},  the constant
$C_{\vk,\, M}$ blows up as $\frac{1}{|1-\vk|}$ when $\vk \to 1$.

\begin{corollary}\label{upper:cor2}
Let $\,u(x,y)\,$ be a function on $\,M\times M\,$ such that
$\,|u(x,y)|\leq C\,d_{x,y}^{\,\vk}\,$ with some nonnegative
constants $\vk$ and $C$, and let $\,K_\la\,$ be the operator defined
by the integral kernel $\,{\mathcal K_\la}(x,y):=u(x,y)\,|\tilde
N_{x,y}(\la)|^2\,$. Then $\,K_\la\,$ maps $L^p(M)$ into $L^p(M)$ and
$$
\|K_{\la}\|_{L^p\to L^p}\ \leq\
\begin{cases}
C_{\vk,\,M}\, (1+\la^{1-\vk}), \,&\vk\ne1\,,\\
C_M\,(1+ |\ln\la|)\,,&\vk=1.
\end{cases}
\qquad\forall p\in[1,\infty].
$$
\end{corollary}

In particular, Corollary~\ref{upper:cor2} implies that the
commutator of the operator given by the integral kernel $\,|\tilde
N_{x,y}(\la)|^2\,$ with the multiplication by a smooth function
$\,f\,$ is bounded in $\,L^p(M)\,$ and its norm is 
$O\,(|\ln\la|)\,$. Indeed, the integral kernel of this commutator
coincides with $\left(f(x)-f(y)\right)\,|\tilde N_{x,y}(\la)|^2\,$,
and $\,|f(x)-f(y)|\leq C\,d_{x,y}\,$ with some $C>0$.

\subsection{Lower bounds}\label{main:lower}
The following theorem shows that our upper estimates cannot be
significantly improved.

\begin{theorem}\label{lower:theorem}
If Condition {\rm\ref{intro:non-conj}} is fulfilled then
\begin{equation}\label{lower1}
\la^{-q-1}\int_0^\la\mu^q\,|\tilde N_{x,y}(\mu)|^p\,d\mu\ \gg\
1\,,\qquad\forall q\geq0\,,\quad p\geq1\,.
\end{equation}
\end{theorem}
In particular, $\,\la^{-1}\int_0^\la|\tilde
N_{x,y}(\mu)|\,d\mu\gg1\,$.

\begin{corollary}\label{lower:cor4}
Let Condition {\rm\ref{intro:non-conj}} be fulfilled, and let $f$ be
a positive function on $\R_+\,$. If there exists a constant
$\,q\geq0\,$ such that
\begin{equation}\label{lower2}
\limsup_{\mu\to+\infty}\left(\mu^{q+1}\,\inf_{\tau\leq\mu}\left(\tau^{-q}f(\tau)\right)\right)\
>\ 0
\end{equation}
then $\,\int f(\mu)\,|\tilde N_{x,y}(\mu)|^p\,\dr\mu=\infty\,$ for
all $p\geq1$. In particular, we have
\begin{equation}\label{lower3}
\int_0^\infty\mu^{-1}\,|\tilde N_{x,y}(\mu)|^p\,\dr\mu\ =\
\infty\,,\qquad\forall p\geq1\,,
\end{equation}
and
\begin{equation}\label{lower4}
\sum_k\mu_k^{-1}\int_{\mu_{k-1}}^{\mu_k}|\tilde
N_{x,y}(\mu)|^p\,\dr\mu\ =\ \infty\,,\qquad\forall p\geq1\,,
\end{equation}
for every increasing sequence $\mu_k\to+\infty$.
\end{corollary}

As follows from \eqref{lower3} with $p=2$, Theorem~\ref{upper:cor1}
fails for $\dr\nu(\la)=(\la+1)^{-1}\dr\la$ on any manifold $M$.

\begin{remark}\label{lower:remark1}
Theorem~\ref{lower:theorem} improves upon \cite[Theorem 1.1.3]{JP}.
It is quite possible that Condition~\ref{intro:non-conj} in this
theorem can be removed: our proof works whenever $\,\si_{x,y}(t)\,$
has a sufficiently strong singularity. It is hard to imagine the
situation where this does not happen; we are not aware of any
counterexamples.
\end{remark}

\subsection{The $\zeta$-function}\label{main:zeta}
The function $\,Z_{x,y}\,$ of complex variable $\,z=t+is\,$ defined
by
\begin{equation}\label{zeta:def}
Z_{x,y}(z)\ :=\ \int_0^\infty \la^{-z}\,\dr N_{x,y}(\la)\ =\
z\int_0^\infty \la^{-z-1}\,N_{x,y}(\la)\,\dr\la
\end{equation}
is said to be the {\em pointwise $\zeta$-function} of the Laplacian.
It is the integral kernel of the pseudodifferential operator
$\Delta^{-\frac{z}2}$. It is well known that $Z_{x,y}(z)$ is an
entire function on $\C$ for all fixed $x\ne y$ (see, for example,
\cite[Theorem 12.1]{Sh}).
\begin{remark}
Further on we call $Z_{x,y}(z)$ simply the $\zeta$-{\it function}.
Note that $Z_{x,y}(z)$ should not be confused with the function
$Z(z)=\Tr \Delta^{-\frac{z}2}$ that is usually referred to as the
$\zeta$-function of the Laplacian.
\end{remark}

The second equality in \eqref{zeta:def} implies that
$\,Z_{x,y}(z)=z\,(\MC\,N_{x,y})(-z)\,$ where $\MC$ is the Mellin
transform, $\,(\MC f)(z):=\int_0^\infty\la^{z-1}\,f(\la)\,\dr\la\,$.
Recall that, for each $t\in\R$, the Mellin transform $(\MC f)(t+is)$
of a distribution $\,f\,$ on $\R_+$ coincides with the inverse
Fourier transform of the distribution $e^{t\mu}f(e^\mu)$ on $\R$
modulo the factor $(2\pi)^{-1}$. The inversion formula reads
$$
f(\la)=(2\pi i)^{-1}\int_{t-i\infty}^{t+i\infty}\la^{-z}(\MC
f)(z)\,\dr z\,,
$$
where the integral is understood in the sense of distributions.

Obviously, $\,e^{t\mu}f(e^\mu)\in L^2(\R)\,$ if and only if
$\,\la^{t-\frac12}f(\la)\in L^2(\R_+)\,$. Therefore
$\,(t+is)^{-1}\,Z_{x,y}(t+is)\in L^2(\R)\,$ if and only if
$\,\la^{-t-\frac12}N_{x,y}(\la)\in L^2(\R_+)\,$ for each fixed $t$,
and
\begin{equation}\label{zeta:inverse}
N_{x,y}(\la)\ =\ (2\pi i)^{-1}\int_{t-i\infty}^{t+i\infty}
\la^{z}\,z^{-1}Z_{x,y}(z)\,\dr z
\end{equation}
in the sense of distributions. In particular, if
$\,(t_0+is)^{-1}\,Z_{x,y}(t_0+is)\in L^2(\R)\,$ for some
$\,t_0\in\R\,$ then $\,(t+is)^{-1}\,Z_{x,y}(t+is)\in L^2(\R)\,$ for
all $\,t>t_0\,$.

In \cite{Randol}, for a surface of constant negative curvature, it
was shown that
\begin{equation}\label{zeta:l2}
(t+i\,\cdot)^{-1}Z_{x,y}(t+i\,\cdot)\in L^2(\R)\,,\qquad\forall
t>\frac{n-1}2\,,
\end{equation}
almost everywhere. By the above, this inclusion is  equivalent to
Theorem~\ref{upper:cor1} with
$\,\dr\nu=(\la+1)^{-1-\eps}\,\dr\la\,$.

Let $\langle s\rangle:=(1+|s|^2)^{1/2}$. In the last section we
shall prove the following two theorems.

\begin{theorem}\label{zeta:theorem1}
$|\,Z_{x,y}(t+is)\,|\ \leq\ \begin{cases} C_t\,,&\text{if}\ \ n<t\,,\\
C_t\left(|s|^{n-t}+d_{x,y}^{\,t-n}\right)\,,&\text{if}\ \ \frac n2\leq t<n\,,\\
C_t\left(|s|^{n-t}+d_{x,y}^{\,t-n}\langle s\rangle^{\frac
n2-t}\right)\,,&\text{if}\ \ t<\frac n2\,.
\end{cases}$
\end{theorem}

\begin{theorem}\label{zeta:theorem2}
For all $\,x\in M\,$ and all $\,\eps>0\,$, we have
\begin{equation}\label{zeta:main}
\int_M\,\frac{|Z_{x,y}(t+is)|^2} {d_{x,y}^{\,2t-n-\eps}+1}\;\dr y\
\leq\ C_{t,\eps}\left(\langle s\rangle^{n-2t}+1\right),
\qquad\forall t\ne\frac{n}2\,.
\end{equation}
\end{theorem}

In view of Fubini's theorem, \eqref{zeta:main} immediately implies
\eqref{zeta:l2} for all $x$ and almost all $y$. However, even the
improved estimate \eqref{zeta:main} does not seem to be sufficient
to obtain our upper bounds for the spectral function.

\begin{remark}\label{z1:remark2}
It is quite possible that Theorems \ref{zeta:theorem1} and
\ref{zeta:theorem2} remain valid for $t=n$ and $t=\frac{n}2$, but
our proof does not work in these cases.
\end{remark}

\subsection{Possible generalizations}\label{main:generalizations}

All the above results can easily be extended to an elliptic
self-adjoint pseudodifferential operator $A$ acting on a compact
manifold without boundary. For such an operator, one has to consider
trajectories of the Hamiltonian flow generated by its principal
symbol instead of geodesics, and to use Fourier integral operators
or the global parametrix constructed in \cite{SV} instead of the
Hadamard representation \eqref{n0:hadamard} for the study of
singularities of $\,\si_{x,y}(t)\,$.

Note that some of our estimates do not require the spectrum to be
discrete. In particular, it may well be possible to extend the
results that do not involve integration over $M$ to the case of
noncompact manifolds. For instance, the functions
$\,\int\rho_1(\la-\mu)\,\dr N_{x,y}(\mu)$ and
$\,\int|\rho_1(\la-\mu)|^2\,\dr N_{x,y}(\mu)$ are the integral
kernels of the operators $\,\rho_1(\la-A)\,$ and
$\,|\rho_1(\la-A)|^2\,$, where $A$ is the restriction of
$\sqrt\Delta$ to the subspace spanned by the eigenfunctions
$\phi_1,\phi_2,\ldots$. Therefore the key equality \eqref{n1:proof1}
is easily obtained by rewriting the obvious operator identity
$\,\left(\rho_1(\la-A)\right)^*\rho_1(\la-A) =|\rho_1(\la-A)|^2\,$
in terms of integral kernels.

It would be also interesting to extend our results to an elliptic
self-adjoint differential operator on a manifold with boundary,
subject to suitable boundary conditions. In this case the role of
geodesics is played by Hamiltonian billiards. One has to consider
interior points $x$ and $y$ and to make appropriate assumptions to
avoid problems with the so-called grazing and dead-end trajectories
(see \cite{SV}).

\section{Almost periodic properties of the spectral function}
\label{sect-conj}

\subsection{Besicovitch almost periodic functions} \label{main:conjectures}
Let $\,p\geq1\,$. Recall that, for a measurable function $f$ on
$\R_+$, its {\it Besicovitch seminorm} $||f||_{{\mathcal B}^p}$  is
defined by
\begin{equation}
||f||_{{\mathcal B}^p}\ :=\ \limsup_{T\to \infty} \left(\frac{1}{T}
\int_0^{T} |f(\la)|^p\,\dr\la \right)^{1/p}
\end{equation}
The space $B^p$ of {\it Besicovitch almost periodic functions} is
defined as the completion of the linear space of all finite
trigonometric sums $\sum_{k=1}^N a_k e^{i\theta_k x}$ with
$a_k\in\C$ and $\theta_k\in\R$ with respect to the Besicovitch
seminorm. Clearly, $\,B^{p_1}\subset B^{p_2}\,$ and
$\,\|f\|_{B^{p_2}}\leq\|f\|_{B^{p_1}}\,$ for all $\,p_1>p_2\,$.

For each real-valued function $\,f\in B^p\,$, there exists a
sequence of real numbers $\,\theta_k\,$  called the {\it
frequencies} of $f$, such that
\begin{equation}
\label{almperexp1}
 \lim_{N\to \infty} ||f-\sum_{k=1}^N
a_k\, \sin (\theta_k x + \phi_k)||_{{\mathcal B}^p}\ =\ 0,
\end{equation}
where the coefficients $\,a_k\in \mathbb{R}$ and the phase shifts
$\phi_k \in \mathbb{R}$ are some constants (see, for example,
\cite{Bes}). If \eqref{almperexp1} holds, we shall write $\,f \sim
\sum a_k\, \sin (\theta_k x + \phi_k)$.

Theorem \ref{upper:cor1} motivates the following
\begin{conj}
\label{besic:weak} On any compact Riemannian manifold $M$, $\|\tilde
N_{x,y}(\la)\|_{{\mathcal B}^2}<\infty$ for each fixed $x \in M$ and
almost all $y\in M$.
\end{conj}
Conjecture \ref{besic:weak} holds for round spheres and flat
$2$-tori. In fact, in Examples \ref{besic:torus} and
\ref{besic:sphere}  we prove more than finiteness of the Besicovitch
seminorm: it is shown that the rescaled spectral function on these
manifolds is $B^2$--almost periodic for each fixed $x$ and almost
all $y$, which implies the finiteness of the Besicovitch seminorm.
In both cases, the set of frequencies coincides with the set
${\mathcal L}_{x,y}$ of lengths of all geodesic segments joining $x$
and $y$.
\begin{remark}
It was proved in \cite{Bl} and \cite{KMS} that the rescaled error
term in Weyl's law has an almost periodic expansion in $B^2$ on
surfaces of revolution and in $B^1$ on Liouville tori. On Zoll
manifolds, the rescaled Weyl remainder (albeit with a different
order of rescaling) has an almost periodic expansion in $B^2$
\cite{Sch}. The frequencies of these expansions are the lengths of
closed geodesics. The spectral function, similarly to the Weyl
remainder, has oscillatory behaviour, and hence it is natural to
study it in the context of almost periodic functions. Moreover, as
indicated by Theorem \ref{upper:cor1}, the order of rescaling in
this case could be chosen {\it universally} for all manifolds of a
given dimension. As was mentioned in Section \ref{intro:disc}, the
lengths of geodesic segments joining $x$ and $y$ are the
singularities of the distribution $\si_{x,y}(t)$,  and hence play
the same role for the spectral function as the lengths of closed
geodesics for the Weyl remainder. The link between the spectral
function and the set of lengths  ${\mathcal L}_{x,y}$ has a natural
interpretation from the viewpoint of the quantum--classical
correspondence.
\end{remark}

\subsection{Spectral function on spheres and tori}
Let us start with a toy example --- the spectral function on the
unit circle $\Sbb^1$:
\begin{equation}
\label{specfuncircle} \tilde N_{x,y}(\la)=
N_{x,y}(\la)=\frac{1}{\pi} \sum_{1 \leq n < \la} \cos(n\,d_{x,y}) =
-\frac{1}{2\pi} + \frac{1}{2\pi} \frac{\sin((\lceil \la \rceil -
\frac{1}{2})d_{x,y})}{\sin(\frac{d_{x,y}}{2})}
\end{equation}
Note that in dimension one no rescaling occurs, and the constant
term $(2\pi)^{-1}$ is subtracted,  because the eigenfunction
corresponding to the zero eigenvalue is excluded in the definition
\eqref{specfundef}.  In higher dimensions the contribution of the
constant term to $\tilde N_{x,y}(\la)$ is negligible in $B^p$ due to
the rescaling.

\begin{example}
\label{example:circle} If $M$ is the unit circle then the spectral
function $\,N_{x,y}(\la)\,$ is $B^2$--almost periodic for all $x \ne
y$ with the set of frequencies $\,{\mathcal
L}_{x,y}\,=\{|d_{x,y}+2\pi k|\,\}_{k\in\Z}\,$, and
\begin{equation}
\label{approxcircle} N_{x,y}(\la) \sim -\frac{1}{2\pi}+
\sum_{-\infty}^{+\infty} \frac{1}{\pi|d_{x,y} + 2\pi k|} \sin(\la
|d_{x,y} + 2\pi k|).
\end{equation}
\end{example}

The next two examples generalize Example \ref{example:circle} to
round spheres and flat two dimensional tori.

\begin{example}
\label{besic:torus} On a flat square $2$-torus $\mathbb T^2 = \R^2 /
(2\pi\Z)^2$, the rescaled spectral function is $B^2$--almost
periodic for all $x \ne y$ and
\begin{equation}
\label{besic:tori1} \tilde N_{x,y}(\la)\ \sim\ \sum_{\eta \in
\mathbb Z^2}
\frac{2\,\sin\left(\la\,|x-y+2\pi\eta|-\frac{\pi}{4}\right)}
{(2\pi)^{3/2}\,|x-y+2\pi\eta|^{3/2}}\,.
\end{equation}
Here $|\cdot|$ denotes the length of a vector in $\mathbb{R}^2$. A
similar formula holds on flat tori corresponding to arbitrary
lattices.
\end{example}

\begin{remark} \label{remark:tori} In dimensions $n \geq3$ the situation is
significantly more complicated. For example, following the argument
of \cite[Corollary 1.4]{Peter} one could show that $||\tilde
N_{x,y}(\lambda)||_{{\mathcal B}^2}=\infty$ on a flat $n$-torus
${\mathbb T}^n=\R^n / (2\pi\Z)^n$ if the vector $ \pi^{-1}(x-y)$ is
{\it rational}. This happens due to unbounded multiplicities in the
set ${\mathcal L}_{x,y}$. However, it does not contradict Conjecture
\ref{besic:weak} because for each fixed $x\in \mathbb{T}^n$,
the vector $\pi^{-1}(x-y)$ has rationally independent coordinates for almost all $y\in
\mathbb{T}^n$.
\end{remark}

Recall that the Morse index of a geodesic segment joining $x$ and
$y$ is the number of points on the segment that are conjugate to
$x$, counted with multiplicities (see \cite[Theorem 15.1]{Mil}). Let
$H(x)$ be the ``reversed'' Heaviside function: $H(x) = 0$ if $x \geq
0$ and $H(x) = 1$ if $x < 0$.

\begin{example}\label{besic:sphere}
Let $M$ be the unit round sphere $\Sbb^n$ of dimension $n\geq 2$,
and let $x \ne y\in \Sbb^n$ be any two non-opposite points. Then
$\tilde N_{x,y}(\la)\in B^2$, and
\begin{equation}\label{besic:sphere1}
\tilde N_{x,y}(\la)\ \sim\
\sum_{k=-\infty}^{+\infty} \frac{2\,\sin\left(\la\,|d_{x,y}+2\pi k|
- \frac{(n-1+2\omega_k)\pi}{4}\right)}{(2\pi)^{\frac{n+1}{2}} (\sin
d_{x,y})^{\frac{n-1}{2}}|d_{x,y}+2\pi k|}
\end{equation}
where $\,\omega_k=(n-1)\left(2\,|k|-H(k)\right)\,$ is the Morse
index of the geodesic segment of length $|d_{x,y}+2\pi k|\,$.
\end{example}

Note that, unlike \eqref{besic:tori1}, the phase shifts in the
expansion \eqref{besic:sphere1} depend on the number of conjugate
points on the geodesic segments.

\subsection{An open problem}
In this section we describe a possible route for generalizing
Examples \ref{besic:torus} and \ref{besic:sphere}. Let $M$ be an
arbitrary compact $n$-dimensional Riemannian manifold.  Let
$\Gamma_{x,y}$ be the set of all geodesic segments joining $x$ and
$y$. For every $\gamma \in \Gamma_{x,y}$, let $l(\gamma)$ be its
length and $\omega(\gamma)$ be its Morse index. As before, set
${\mathcal L}_{x,y}=\{l(\gamma)|\, \gamma \in \Gamma_{x,y}\}$. Along
each geodesic segment $\gamma \in \Gamma_{x,y}$, consider the matrix
Jacobi equation $A''+R A =0$, where the coefficient $R$ is defined
in terms of the Riemann curvature tensor and the parallel transport
along $\gamma$  (see \cite[p. 104]{Chavel}). We assume that $\gamma$
is naturally parametrized, and $A$ satisfies the initial conditions
$A(0,\xi_\gamma)=0$, $A'(0,\xi_\gamma)=1$, where $\xi_\gamma$ is the
unit tangent vector to $\gamma$ at the point $x$. Set
$a(\gamma)=|\det A(l(\gamma),\xi_\gamma)|$. For example,
$a(\gamma)=l(\gamma)^{n-1}$ on a flat square $n$-torus and
$a(\gamma)=|\sin l(\gamma)|^{n-1}$ on a round $n$-sphere. In
dimension two, $a(\gamma)=|J(l(\gamma))|$, where $J(t)$ is the
orthogonal Jacobi field along $\gamma$ with the initial conditions
$J(0)=0$ and $J'(0)=1$.

\begin{problem}
\label{besic:conj3} Let $M$ be a compact $n$-dimensional Riemannian
manifold. Is it true that for all $x \in M$ and almost all $y\in M$
the rescaled spectral function $\tilde N_{x,y}(\la)$ has an almost
periodic expansion
\begin{equation}
\label{almperexp} \tilde N_{x,y}(\la) \sim
\frac{2}{(2\pi)^{\frac{n+1}{2}}} \sum_{\gamma \in \Gamma_{x,y}}
\frac {\sin(\la l(\gamma) - \frac{(n-1)\pi}{4}-\omega(\gamma)
\frac{\pi}{2})}{l(\gamma) \sqrt{a(\gamma)}}
\end{equation}
in $B^p$ for some $p \geq1$?
\end{problem}

Let us show that the expansion \eqref{almperexp} is well-defined if
the points $x,y \in M$ are not conjugate along any geodesic joining
them (by \cite[Corollary 18.2]{Mil}, this condition is satisfied for
any
fixed $x\in M$ and almost all $y \in M$).   
Indeed, if the points $x, y$ are not conjugate along any geodesic,
$a(\gamma) \ne 0$ for any $\gamma$, the set of lengths ${\mathcal
L}_{x,y}$ is discrete, and each element has finite multiplicity,
i.e. appears in ${\mathcal L}_{x,y}$ at most a finite number of
times \cite[Theorem 16.3]{Mil}. This implies, in particular, that
the set ${\mathcal L}_{x,y}$ is infinite, because for any two points
on a compact manifold there exists an infinite number of geodesic
segments joining them \cite{Serre}. If $M$ has no conjugate points,
it is easy to check that \eqref{almperexp} agrees with \cite[formula
(5.1.3)]{JP}.

\begin{remark}
If \eqref{almperexp} does hold for some $\,p\geq1\,$ on a given
manifold, it would be interesting to determine the maximal possible
value of $\,p\,$. For instance, it is quite likely that for round
spheres one can take any $p\geq1$.
Note that if the almost periodic expansion is valid for some $p>1$,
then the Fourier coefficients of \eqref{almperexp} lie in $l_q$ for
$q=\max(2,\frac{p}{p-1})$ (\cite[section 4]{ABI}). This can be
viewed as a dynamical condition on the manifold (see \cite[Section
3.1]{Pat} for some related results), and it is not clear whether it
always holds. At the same time, even for $p=1$,  a positive answer
to Problem \ref{besic:conj3} provides a lot of useful information
about the spectral function; in particular, it implies that the
rescaled spectral function has a {\it limit distribution} (see
\cite[Appendix II]{KMS}). Understanding the properties of this
distribution on a given manifold is a problem of independent
interest.
\end{remark}

\section{Proofs of the upper bounds}
\label{proof-upper}
\subsection{Auxiliary functions}
\label{proof-upper:auxiliary} Let us fix a real-valued even rapidly
decreasing function $\,\rho\in C^\infty(\R)\,$ satisfying the
following condition.
\begin{condition}\label{rho}
$\,\hat\rho\in C_0^\infty(\R)\,$, $\,\hat\rho\equiv1\,$ in a
neighbourhood of the origin and $\,\supp\hat\rho\subset(-\de,\de)\,$
for some $\de>0\,$.
\end{condition}

Condition~\ref{rho} implies that
$\,\int_0^\infty\rho(\tau)\,\dr\tau=\frac12\,$ and
$\int_0^\infty\tau^k\,\rho(\tau)\,\dr\tau=0$ for all $k=1,2,\ldots$
Let
$\,\rho_1(\tau):=\sign\tau\,\int_{|\tau|}^\infty\rho(\mu)\,\dr\mu\,$
if $\,\tau\ne0\,$, and $\,\rho_1(0):=-\frac12\,$. The function
$\rho_1$ is rapidly decreasing, odd and infinitely differentiable
outside the origin $\tau=0$, so that
$\,\frac{\dr\hfill}{\dr\tau}\,\rho_1(\tau)=-\rho(\tau)\,$ for all
$\tau\ne0\,$. It has a jump at the origin; in view of
Condition~\ref{rho}, $\rho_1(0)=\rho_1(-0)=-\rho_1(+0)=-\frac12\,$.

Let
\begin{equation}\label{n0,n1}
N_{x,y;0}:=\rho*N_{x,y}\,,\qquad N_{x,y;1}:=N_{x,y}\;-\;N_{x,y;0}\,.
\end{equation}
The following elementary lemma is a slight variation of \cite[Lemma
1.2]{S2}.

\begin{lemma}\label{n1:conv}
$N_{x,y;1}(\la)=\int\rho_1(\la-\mu)\,\dr N_{x,y}(\mu)\,$ for all
$\la\in\R$.
\end{lemma}

\begin{proof}
If $\la$ is not an eigenvalue then, integrating by parts, we
immediately obtain
\begin{multline*}
\int\rho_1(\la-\mu)\,\dr N_{x,y}(\mu)
=\int_{-\infty}^\la\rho_1(\la-\mu)\,\dr N_{x,y}(\mu)+
\int_\la^\infty\rho_1(\la-\mu)\,\dr N_{x,y}(\mu)\\
=\ \left(\rho_1(+0)-\rho_1(-0)\right)N_{x,y}(\la)
-\int\rho(\la-\mu)\,N_{x,y}(\mu)\,\dr\mu\,,
\end{multline*}
where the right hand side coincides with $N_{x,y;1}(\la)$. If
$\la=\la_j$ then the same equality holds for the function
$$
N_{x,y}^{(j)}(\la)\ :=\ N_{x,y}(\la)-\left(N_{x,y}(\la_j+0)
-N_{x,y}(\la_j-0)\right)\,\chi_j(\la)\,,
$$
where $\,\chi_j(\la)\,$ is the characteristic function of the
interval $\,(\la_j,\infty)\,$. Since
\begin{multline*}
\int\rho_1(\la_j-\mu)\,\dr\chi_j(\mu)\ =\ \rho_1(0)\ =\ \rho_1(-0)\
=\ -\int_0^\infty\rho(\tau)\,\dr\tau\\
=\ -\int\rho(\la_j-\mu)\,\chi_j(\mu)\,\dr\mu\ =\
\chi(\la_j)-\int\rho(\la_j-\mu)\,\chi_j(\mu)\,\dr\mu\,,
\end{multline*}
the lemma remains valid when $\la$ is an eigenvalue.
\end{proof}

\begin{remark}\label{rho1}
The above lemma turns out to be very useful for obtaining estimates
of the spectral and counting functions. Usually, the singularities
of $\,\si_{x,y}(t)\,$ for small values of $\,t\,$ can be described
explicitly. Then, taking the inverse Fourier transform, one obtains
full asymptotic expansion of $\,N_{x,y;0}(\la)\,$. The asymptotic
behaviour of $\,N_{x,y;1}(\la)\,$ is determined by nonzero
singularities of $\,\si_{x,y}(t)\,$, which are much more difficult
to study. According to Lemma~\ref{n1:conv}, the function
$\,N_{x,y;1}(\la)\,$ can also be written as a convolution. The main
technical problem is that the function $\,\rho_1\,$ has a jump, and
therefore the straightforward integration by parts does not yield
any new results. However, as we shall see in the next subsection,
this jump disappears when we square $\,N_{x,y;1}(\la)\,$ and
integrate over $\,y\,$. Note also that $\,|\rho_1|\,$ is estimated
by a smooth function whose Fourier transform has a compact support.
This observation allows one to simplify and refine the well known
Fourier Tauberian Theorems (see \cite{S2}).
\end{remark}

\subsection{Upper bounds for $N_{x,y;1}(\la)\,$}\label{proof-upper:n1}
Since the eigenfunction $\,\phi_j\,$ are orthogonal in $L^2(M)$,
Lemma~\ref{n1:conv} implies that
\begin{multline}\label{n1:proof1}
\int_M|N_{x,y;1}(\la)|^2\,\dr y\ =\
\int_M|\int\rho_1(\la-\mu)\,\dr N_{x,y}(\mu)|^2\,\dr y\\
=\ \int_M\iint\rho_1(\la-\mu)\,\rho_1(\la-\tau)\,\dr
N_{x,y}(\mu)\,\dr N_{x,y}(\tau)\,\dr y\ =\ \int
|\rho_1(\la-\mu)|^2\,\dr N_{x,x}(\mu)\,
\end{multline}
(an operator interpretation of this equality has been given in the
subsection \ref{main:generalizations}). The function
$|\rho_1(\tau)|^2$ is continuous and infinitely differentiable
outside the origin. Integrating by parts and changing variables in
the right hand side, we obtain
$$
\int_M|N_{x,y;1}(\la)|^2\,\dr y=\int|\rho_1(\la-\mu)|^2\,\dr
N_{x,x}(\mu) =\int\tilde\rho_1(\mu)\,N_{x,x}(\la-\mu)\,\dr\mu\,,
$$
where
$\,\tilde\rho_1(\tau):=\frac{\dr\hfill}{\dr\tau}\,|\rho_1(\tau)|^2$
is a rapidly decreasing odd function. Since $\,\tilde\rho_1\,$ is
odd,
$$
\int\tilde\rho_1(\mu)\,N_{x,x}(\la-\mu)\,\dr\mu
=\int_0^\infty\tilde\rho_1(\mu)\left(N_{x,x}(\la-\mu)-N_{x,x}(\la+\mu)\right)\dr\mu\,.
$$
Substituting \eqref{intro:bound1} in the right hand side and
estimating
$$
|(\la-\mu)^n-(\la+\mu)^n|\ \leq\
C\,(|\mu|\,\la^{n-1}+|\mu|^n)\,,\quad|\la\pm\mu|^{n-1}\ \leq\
C\,(\la^{n-1}+|\mu|^{n-1})\,,
$$
we see that
\begin{equation}\label{n1:proof2}
\int_M|N_{x,y;1}(\la)|^2\,\dr y\ \leq\
C\,(\la^{n-1}+1)\,,\qquad\forall\la>0\,,
\end{equation}
where the constant $C$ depends only on the dimension, the remainder
term in the Weyl formula~\eqref{intro:bound1} and the auxiliary
function $\rho$.

\subsection{Upper bounds for $N_{x,y;0}(\la)\,$}\label{proof-upper:n0}
Since
$\,\frac{\dr}{\dr\la}N_{x,y;0}(\la)=\int\rho(\la-\mu)\,\dr
N_{x,y}(\mu)\,$ and
$$
\int\rho(\la-\mu)\,\dr
N_{x,y}(\mu)=\rho*N'_{x,y}(\la)=\left(\hat\rho\,\si_{x,y}\right)^\vee(\la)\,,
$$
Condition~\ref{rho} implies that the asymptotic behaviour of
$N_{x,y;0}(\la)$ for large $\la$ is determined by the singularities
of $\si_{x,y}$ on $\,\supp\hat\rho\,$. For second order differential
operators, it is slightly more convenient to deal with the cosine
Fourier transform $\,e_{x,y}(t):=\int\cos(t\la)\,\dr
N_{x,y}(\la)\,$. The distribution $e_{x,y}$ coincides with the
fundamental solution of the wave equation. For all sufficiently
small $t$, it admits the following Hadamard representation
\begin{equation}\label{n0:hadamard}
e_{x,y}(t)\ =\ |t|\,\sum_{j=0}^\infty
u_j(x,y)\,\frac{(t^2-d_{x,y}^{\,2})_+^{j-\frac{n+1}2}}
{\Gamma(j-\frac{n+1}2+1)}\mod
C^\infty
\end{equation}
where $u_j(x,y)$ are some smooth functions (see \cite{Ba} or
\cite[Proposition 27]{Ber}). Let us assume that $\de$ in
Condition~\ref{rho} is small enough so that \eqref{n0:hadamard}
holds on the interval $(-\de,\de)$ (and, consequently, on
$\,\supp\hat\rho\,$).

Denote $\,\Psi_{x,y}(\la):=N_{x,y}(\la)- N_{x,y}(-\la)\,$. Changing
variables and taking into account that the function $\rho$ is even,
we see that
$$
\rho*\Psi_{x,y}(\la)=N_{x,y;0}(\la)-N_{x,y;0}(-\la)\,.
$$
Since $N_{x,y}(\la)=0$ for $\la<0$ and $\rho$ is a rapidly
decreasing function, $\,N_{x,y;0}(-\la)\,$ vanishes faster than any
power of $\la$ as $\la\to+\infty$. Thus we have
\begin{equation}\label{n0-Psi}
N_{x,y;0}(\la)\ =\ \rho*\Psi_{x,y}(\la)\ +\
o(\la^{-m})\,,\qquad\forall m\in\R_+\,.
\end{equation}

Integrating by parts, we obtain
\begin{equation}\label{e-Psi}
e_{x,y}(t)\ =\ t\int\sin(t\la)\, N_{x,y}(\la)\,\dr\la\ =\
-(2i)^{-1}\,t\,\hat\Psi_{x,y}(t)\,.
\end{equation}
In view of \eqref{n0:hadamard} and \eqref{e-Psi},
\begin{equation}\label{Psi}
\hat\Psi_{x,y}(t)\ =\ -2i\,\sign t\,\sum_{j=0}^\infty
u_j(x,y)\,\frac{(t^2-d_{x,y}^{\,2})_+^{j-\frac{n+1}2}}{\Gamma(j-\frac{n+1}2+1)}\mod
C^\infty
\end{equation}
on $\,\supp\hat\rho\,$. We have
\begin{multline}\label{f-transform}
\int_{|t|>d_{x,y}} e^{it\la}\,\sign t\,
\frac{(t^2-d_{x,y}^2)^{p-\frac12}}{\Gamma\left(p+\frac12\right)}\,\dr
t \ =\ 2i\int_{d_{x,y}}^\infty\sin(t\la)\,
\frac{(t^2-d_{x,y}^{\,2})^{p-\frac12}}{\Gamma\left(p+\frac12\right)}\,\dr t\\
=\ i\sqrt\pi\,\left(\frac{2\,d_{x,y}}{\la}\right)^p
J_{-p}(d_{x,y}\la)\,,
\end{multline}
(see, for example, \cite[formula 3.771(7)]{GR} or \cite[Chapter II,
Section 2.5]{GS}). Therefore, taking the inverse Fourier transform
of each term in the right-hand side of \eqref{Psi}, we obtain an
asymptotic series
\begin{equation}\label{asymp}
\sum_{j=0}^\infty w_j(x,y)\,d_{x,y}^{\,j-\frac{n}2}\,
\la^{\frac{n}2-j}\, J_{\frac{n}2-j}\left(d_{x,y}\,\la\right)
\end{equation}
with some bounded functions $w_j(x,y)$. Let $\Psi_{0;x,y}(\la)$ be
an arbitrary function such that $\,\Psi_{0;x,y}(\la)\sim
\sum_{j=0}^\infty w_j(x,y)\,d_{x,y}^{\,j-\frac{n}2}\,
\la^{\frac{n}2-j}\, J_{\frac{n}2-j}\left(d_{x,y}\,\la\right)\,$ as
$\,\la\to+\infty\,$ uniformly with respect to $x,y\in M$. By the
above,
\begin{equation}\label{Psi-Psi0}
\rho*\Psi_{x,y}(\la)-\rho*\Psi_{0;x,y}(\la)\ =\ O(\la^{-\infty})\,.
\end{equation}

Recall that the Bessel function $\,J_p(\tau)\,$ is estimated by
$\,C\,\tau^p\,$ in a neighbourhood of the origin and does not exceed
$\,C\,\tau^{-\frac12}\,$ for large $\tau$ (see \cite[8.440 and
8.451(1)]{GR}). Therefore
\begin{equation}\label{bessel}
|d_{x,y}^{\,-p}\, \la^{\frac12}\, J_p\left(d_{x,y}\,\la\right)|\
\leq\
\begin{cases}
C\,\la^{p+\frac12}\,,&\la\leq d_{x,y}^{\,-1}\,,\\
C\,d_{x,y}^{\,-p-\frac{1}2}\,,&\la> d_{x,y}^{\,-1}\,,
\end{cases}
\qquad\forall p\in\R\,.
\end{equation}
The inequalities \eqref{bessel} imply that the terms in the
expansion \eqref{asymp} are bounded by
$\,C_j\,d_{x,y}^{\,j-\frac{n+1}2}\,\la^{\frac{n-1}2-j}$ for $\la\geq
d_{x,y}^{\,-1}\,$ and are estimated by $\,C_j\,\la^{n-2j}$ for
$\la\in(\eps,d_{x,y}^{\,-1})\,$, where $\eps$ is an arbitrary
positive constant. Thus, if $\la>\eps$ then
$$
\left|w_j(x,y)\,d_{x,y}^{\,j-\frac{n}2}\, \la^{\frac{n}2-j}\,
J_{\frac{n}2-j}\left(d_{x,y}\,\la\right)\right|\ \leq\
C_j\left(d_{x,y}^{\,j-\frac{n+1}2}+1\right)\la^{\frac{n-1}2-j}
$$
for all $j=0,1,2,\ldots$ and, consequently,
\begin{equation}\label{Psi0-2}
\left|\Psi_{0;x,y}(\la)-w_0(x,y)\,d_{x,y}^{\,-\frac{n}2}\,
\la^{\frac{n}2}\, J_{\frac{n}2}\left(d_{x,y}\,\la\right)\right|\
\leq\ C\left(d_{x,y}^{\,\frac{1-n}2}+1\right)\la^{\frac{n-3}2}\,.
\end{equation}

By direct calculation, $\,w_0(x,y)= 2^{-\frac{n}{2}}
\pi^{-\frac{1}{2}} u_0(x,y)\,$. We have
$$
u_0(x,y)\ =\
\pi^{\frac{1-n}{2}}\left(\det\{g_{jk}(x,y)\}\right)^{-1/4}\,,
$$
where $\{g_{jk}(x,y)\}$ is the metric tensor in geodesic normal
coordinates with origin $x$ (\cite[formula (3.1.1)]{JP}), so that
$w_0(x,y) = (2\pi)^{-\frac{n}{2}} + O(d_{x,y}^{\,2})$ (see \cite[p.
101]{Ros}). From this estimate and \eqref{Psi0-2} it follows that
$$
\int_M\left|\Psi_{0;x,y}(\la)-(2\pi)^{-\frac{n}{2}}\,d_{x,y}^{\,-\frac{n}2}\,
\la^{\frac{n}2}\,
J_{\frac{n}2}\left(d_{x,y}\,\la\right)\right|^2\,\dr y\ \leq\
C_\eps\,\la^{n-3}\int_M\left(d_{x,y}^{\,1-n}+1\right)\dr y
$$
for all $\la>\eps$, where $\eps$ is an arbitrary positive number and
$C_\eps$ is a constant depending only on $\eps$, the auxiliary
function $\rho$ and the geometry of $M$. Finally, since the function
$\rho$ is rapidly decreasing, the above inequality, \eqref{n0-Psi}
and \eqref{Psi-Psi0} imply that
\begin{multline}\label{n0}
\int_M\left|N_{x,y;0}(\la)-(2\pi)^{-\frac{n}{2}}\,d_{x,y}^{\,-\frac{n}2}\,
\la^{\frac{n}2}\,
J_{\frac{n}2}\left(d_{x,y}\,\la\right)\right|^2\,\dr y\\
\leq\ C_\eps\,\la^{n-3}\int_M\left(d_{x,y}^{\,1-n}+1\right)\dr
y\,,\qquad\forall\la\geq\eps>0\,,
\end{multline}
where $C_\eps$ is another constant depending only on $\eps$, $\rho$
and the geometry of $M$.

\begin{remark}\label{n0:remark1}
The above proof is a slight modification of arguments in
\cite{Ba}.
\end{remark}

\subsection{Proof of Theorem~\ref{upper:theorem}} The theorem
follows from \eqref{n0}, \eqref{n1:proof2} and \eqref{specrn}.

\subsection{Proof of Theorem~\ref{upper:cor3}}\label{proof-upper:cor3}
By \eqref{bessel}, there exists a constant $C_1$ such that
$$
\left|d_{x,y}^{\,-\frac{n}2}\, \mu^{\frac12}\,
J_{\frac{n}2}\left(d_{x,y}\,\mu\right)\right|\ \leq\
C_1\,d_{x,y}^{\,-\frac{n+1}2}.
$$
If $\,C\,> \sqrt{2}C_1$ then the above inequality and
\eqref{upper:int0} imply that
\begin{multline*}
2\,C_M\ \geq\ 2\int_{\Omega_x(\la,\mu)}\left|\tilde
N_{x,y}(\la)-d_{x,y}^{\,-\frac{n}2}\, \mu^{\frac12}\,
J_{\frac{n}2}\left(d_{x,y}\,\mu\right)\right|^2\,\dr y\\
\geq\ \int_{\Omega_x(\la,\mu)}\left(|\tilde
N_{x,y}(\la)|^2-2\,C_1^2\,d_{x,y}^{\,-n-1}\right)\,\dr y\\ \geq\
\int_{\Omega_x(\la,\mu)}\left(C^2\,\mu^2+
(C^2-2\,C_1^2)\,d_{x,y}^{\,-n-1}\right)\,\dr y\ \geq\
C^2\,\mu^2\int_{\Omega_x(\la,\mu)}\,\dr y\,.
\end{multline*}

\subsection{Proof of Corollary~\ref{upper:cor4}}\label{proof-upper:cor4}
If $\,y\ne x\,$ then the inequalities \eqref{upper:seq1} imply that
$\,y\in\bigcap_{k\geq N}\Omega_x(\tau_k,(2C)^{-1}\mu_k)\,$ for all
sufficiently large $\,N\,$. By Theorem~\ref{upper:cor3}, the measure
of this intersection is zero.

If $\,\sum_k\mu_k^{-2}<\infty\,$ then, according to
Theorem~\ref{upper:cor3}, the sum of measures of the sets
$\,\Omega_x(\tau_k,t\mu_k)\,$ is finite for all $t>0$. By the
Borel--Cantelli lemma, in this case almost every point $y\in M$
belongs only to finitely many sets $\,\Omega_x(\tau_k,t\mu_k)\,$.
This implies that $\,\limsup_{k\to\infty}\left(\mu_k^{-1}\,\tilde
N_{x,y}(\tau_k)\right)<t\,$ almost everywhere. Letting $\,t\to0\,$,
we obtain \eqref{upper:seq2}.

\subsection{Proof of Theorem~\ref{upper:cor0}}\label{proof-upper:cor0}
Let $\,d_M:=\diam M\,$. We have
\begin{multline*}
\int_M d_{x,y}^{\,\vk} \left|d_{x,y}^{\,-\frac{n}2}\,
\la^{\frac12}\, J_{\frac{n}2}\left(d_{x,y}\,\la\right)\right|^2\dr
y\ \leq\ C\,\la\int_M d_{x,y}^{\,\vk-n}\,
J_{\frac{n}2}^2\left(d_{x,y}\,\la\right)\dr y\\
\ \leq\ C\,\la\int_0^{d_M}r^{\,\vk-1}\, J_{\frac{n}2}^2(r\,\la)\,\dr
r\ =\ C\,\la^{1-\vk}\int_0^{\la\,d_M}r^{\,\vk-1}\,
J_{\frac{n}2}^2(r)\,\dr r\,.
\end{multline*}
Even if spherical coordinates centred at $x$ do not exist globally
on $M$, it is still possible to use the pull-back of the metric
under the exponential map to change the integration variable to $r$.
Since the Bessel function $\,J_{\frac{n}2}\,$ is bounded by
$\,C\,r^{\frac{n}2}\,$ in a neighbourhood of the origin and does not
exceed $\,C\,r^{-\frac12}\,$ for large $\,r\,$, the right hand side
is estimated by
$\,C\,\la^{1-\vk}\left((n+\vk)^{-1}+\int_1^{\la\,d_M}r^{\,\vk-2}\,\dr
r\right)\,$ if $\la\,d_M> 1$. Thus we obtain
$$
\int_M d_{x,y}^{\,\vk} \left|w_0(x,y)\,d_{x,y}^{\,-\frac{n}2}\,
\la^{\frac12}\, J_{\frac{n}2}\left(d_{x,y}\,\la\right)\right|^2\dr
y\ \leq\
\begin{cases}
C_{\vk}\,(\la^{1-\vk}+d_M^{\vk-1})\,,&\vk\ne1\,,\\
C\,|\ln\la|\,,&\vk=1\,.
\end{cases}
$$
with $C_{\vk} = C |\vk-1|^{-1}$. This estimate together with
\eqref{upper:int0} (multiplied by $d_M^{\vk}$ on both sides) imply
\eqref{upper:int1} by triangle inequality.

\subsection{Proof of Corollary~\ref{upper:cor2}}\label{proof-upper:cor2}
Let $g_\vk(\la):=C_{\vk,\,M}\, (1+\la^{1-\vk})$ if $\vk\ne1$ and
$g_1(\la):=C_M(1+|\ln\la|)$. By \eqref{upper:int1},
$$
|\int_M{\mathcal K}_\la(x,y)\,u(y)\,\dr y\,|\ \leq\
g_\vk(\la)\,\|u\|_{L^\infty}
$$
for all $\,u\in L^\infty(M)\,$ and
\begin{multline*}
\int_M \left|\int_M{\mathcal K}_\la(x,y)\,v(y)\,\dr y\,\right|\,\dr
x\leq\int_M\left(\int_M{\mathcal K}_\la(x,y)\,\dr
x\right)|v(y)|\,\dr y\leq g_\vk(\la)\,\|v\|_{L^1}
\end{multline*}
for all $v\in L^1(M)\,$. These estimates imply the required result
for $p=\infty$ and $p=1$ respectively. For other values of $p$ the
result follows from the Riesz interpolation theorem (see
\cite[Theorem 0.1.13]{Sogge}).

\subsection{Proof of Theorem~\ref{upper:cor1}}\label{proof-upper:cor1}
If $\vk>1$ then, in view of \eqref{upper:int1},
$$
\int_0^\infty\int_M d_{x,y}^{\,\vk}\,|\tilde N_{x,y}(\la)|^2\,\dr
y\,\dr\nu(\la)\ <\ \infty\,.
$$
Therefore, by Fubini's theorem, the integral $\,\int_0^\infty|\tilde
N_{x,y}(\la)|^2\,\dr\nu(\la)\,$ is finite for almost all $y\in M$.

\section{Proofs of the lower bounds}\label{proof-lower}
\subsection{Proof of Theorem~\ref{lower:theorem}}\label{proof-lower:theorem}
Let the points $x,y \in M$ be not conjugate along any shortest
geodesic segment joining them. Then, as follows from \cite[Theorem
16.2]{Mil}, the number of shortest geodesic segments joining $x$ and
$y$ is finite, and there exists $\eps>0$ such that
$(d_{x,y}-\eps,d_{x,y}+\eps) \cap {\mathcal L}_{x,y}=d_{x,y}$. Set
$\psi(\la)=\exp(i\la\,d_{x,y})\,\rho(\la)$, where $\rho$ satisfies
Condition \ref{rho}. Then $\hat\psi(d_{x,y})=1$, and one could
choose $\rho$ in such a way that ${\rm supp}\,\, \hat\psi \subset
(d_{x,y}-\eps, d_{x,y}+\eps)$. By \cite[Theorem 4.2]{LSV} (see also
\cite[Proposition 3.3.6]{JP}) we have:
\begin{equation}\label{lower:proof1}
\left(\hat\psi\,\si_{x,y}\right)^\vee(\mu)\ =\
\int\psi(\mu-\tau)\,\dr N_{x,y}(\tau)\ =\ C
\,e^{-i\,(d_{x,y}\mu+\phi)}\,\mu^{\frac{n-1}2}+O\left(\mu^{\frac{n-3}2}\right),
\end{equation}
where $\phi$ is the phase shift depending only on dimension of $M$
(note that the Morse index of any shortest geodesic is zero). Since
$$
\int|\psi'(\mu-\tau)|\,|N_{x,y}(\tau)|\,\dr\tau\ \geq\
|\int\psi'(\mu-\tau)\,N_{x,y}(\tau)\,\dr\tau|\ =
|\int\psi(\mu-\tau)\,\dr N_{x,y}(\tau)|\,,
$$
\eqref{lower:proof1} implies that
$$
\int\mu^q\,|\psi'(\mu-\tau)|\,|N_{x,y}(\tau)|\,\dr\tau\ \geq\ C
\,\mu^{\frac{n-1}2+q}+O\left(\mu^{\frac{n-3}2+q}\right)\,,\qquad\forall
q\ge0\,.
$$

The rest of the proof is similar to that of \cite[Lemma 5.1]{Sar}
Integrating the above inequality and taking into account the
estimate
\begin{equation}\label{lower:proof2}
\int_0^{\frac\la2}\,\mu^{\frac{n-1}2+q}\,\dr\mu\ \gg\
\la^{\frac{n+1}2+q}\,,
\end{equation}
we obtain
\begin{equation}\label{lower:proof3}
\int\tilde\psi(\la,\tau)\,|N_{x,y}(\tau)|\,\dr\tau\ \geq\
C\,\la^{\frac{n+1}2+q}+O\left(\la^{\frac{n-1}2+q}\right),
\end{equation}
where
$\tilde\psi(\la,\tau):=\int_0^{\frac\la2}\mu^q\,|\psi'(\mu-\tau)|\dr\mu\,$.
Since $\psi'$ is a rapidly decreasing function, we have
$$
\tilde\psi(\la,\tau)\ \leq\
C\int_0^{\frac\la2}\left(\tau^q+|\mu-\tau|^q\right)|\psi'(\mu-\tau)|\dr\mu\
\leq\ C\,(\tau^q+1)\,,\qquad\forall\tau\in\R_+\,,
$$
and $\,\tilde\psi(\la,\tau)\ \leq\
C_j\left(\la\,|\la-\tau|\right)^{-j}\,$ for all $\,\tau\geq\la\,$
and $\,j=1,2,\ldots$ These estimates imply, respectively, that
\begin{equation}\label{lower:proof4}
\int_0^\la\tilde\psi(\la,\tau)\,|N_{x,y}(\tau)|\,\dr\tau\leq\
C\int_0^\la\,\tau^q\,|N_{x,y}(\tau)|\,\dr\tau
\end{equation}
and
\begin{equation}\label{lower:proof5}
\int_\la^\infty\tilde\psi(\la,\tau)\,|N_{x,y}(\tau)|\,\dr\tau\ \leq\
C\int_\la^\infty\tilde\psi(\la,\tau)\,\tau^n\,\dr\tau\ =\ O(1)\,.
\end{equation}
Putting together \eqref{lower:proof3}--\eqref{lower:proof5}, we see
that
$$
\la^{-q-1}\int_0^\la\,\tau^q\,|\tilde N_{x,y}(\tau)|\,\dr\tau\ \gg\
\la^{-\frac{n+1}2-q}\int_0^\la\,\tau^q\,|N_{x,y}(\tau)|\,\dr\tau\
\gg\ 1\,.
$$
Now \eqref{lower1} with $p>1$ follows from Jensen's inequality. \qed

\begin{remark}\label{lower:remark2}
The inequality \eqref{lower:proof1} holds for $N_{x,x}(\la)$ under
the assumption that $x$ is not conjugate to itself along any
shortest geodesic loop. Therefore, similar lower bounds can be
proved for the oscillatory error term $R^{osc}_x(\la)$ introduced in
\cite[section 1.2]{JP}. Note that in dimension two this term
coincides with the usual pointwise remainder in Weyl's law.

\end{remark}
\subsection{Proof of Corollary~\ref{lower:cor4}}\label{proof-lower:cor4}
Condition \eqref{lower2} implies that there is a sequence of points
$\,\mu_k\to+\infty\,$ and a positive constant $\,C\,$ such that
$$
\mu_k^{q+1}\mu^{-q}f(\mu)\ \geq\ C\,,\qquad\forall\mu\leq\mu_k\,.
$$
Let $\chi_k(\mu)\,$ be the characteristic functions of the intervals
$(0,\mu_k]\,$, $k=1,2,\ldots$ By the above, the functions
$\,g_k(\mu):=\mu_k^{-q-1}\chi_k(\mu)\,\mu^q\left(f(\mu)\right)^{-1}\,$
are uniformly bounded by a constant. Obviously, $\,g_k(\mu)\to0\,$
as $k\to\infty$ for each fixed $\mu$. If the integral $\,\int
f(\mu)\,|N_{x,y}(\mu)|^p\,\dr\mu\,$ were finite then, by the
Lebesgue dominated convergence theorem, we would have
$$
\mu_k^{-q-1}\int_0^{\mu_k}\mu^q\,|N_{x,y}(\mu)|^p\,\dr\mu =\int
g_k(\mu)\,f(\mu)\,|N_{x,y}(\mu)|^p\,\dr\mu\ \to\ 0\,,\qquad
k\to\infty\,.
$$
However, this contradicts to \eqref{lower1}.

The estimates \eqref{lower3} and \eqref{lower4} are obtained by
taking $f(\mu):=\mu^{-1}$ and
$f(\mu):=\mu_k^{-1}\left(\chi_k(\mu)-\chi_{k-1}(\mu)\right)$
respectively.

\section{Examples}\label{examples}
\subsection{Circle and $2$-tori}\label{examples:tori}
In this subsection we justify Examples~\ref{example:circle} and
\ref{besic:torus}. Let $n=1,2$ and let $\Gamma$ be a lattice in
$\R^n$. Consider the flat torus $\mathbb T^n = \R^n / \Gamma$. If
$n=1$ one can assume that $\Gamma=2\pi \mathbb{Z}$, and thus
$\mathbb{T}^1 \cong \mathbb{S}^1$. We have
$$
N_{x,y}(\la)\ =\ \frac{1}{\Vol(\Gamma)} \sum_{\xi\in \Gamma^{*}}
e^{2\pi i\langle x-y,\xi\rangle}\,\chi_{B^n}(2 \pi \xi/\la)\,
$$
where $\,\chi_{B^n}\,$ is the characteristic function of the unit
ball, $\Vol(\Gamma)$ is the volume of the torus and
$\Gamma^{*}=\{\xi \in
\mathbb{R}^n\,:\,(\xi,\eta)\in\Z\,,\,\forall\eta \in \Gamma\}$ is
the dual lattice (see \cite[p. 29]{Chavel2}. We are slightly abusing
notation here, since we added to the spectral function the constant
term corresponding to $\xi = 0$. Note that $\,N_{x,y}=N_{x-y,0}\,$,
hence we can assume without loss of generality that $\,y=0\,$.

Let $w(\xi)\in C_0^\infty(\R^n)$ be a nonnegative function depending
only on $|\xi|$, such that $\int w(\xi)\,\dr\xi=1\,$. Denote
$\,w_\eps(\xi):=\eps^{-n}\,w(\eps^{-1}\xi)\,$ for $\eps > 0$. By the
Poisson summation formula,
\begin{multline*}
N_{x,0}^{(\eps)}(\la)\ :=\ \frac{1}{\Vol(\Gamma)} \sum_{\xi\in
\Gamma^{*}}
e^{2\pi i \langle x,\xi\rangle}(w_\eps*\chi_{B^n})(2 \pi \xi/\la)\\
= \sum_{\eta\in \Gamma} \int_{\R^n} e^{2 \pi i\langle x + \eta,\xi\rangle}(w_\eps*\chi_{B^n})(2  \pi \xi/\la)\,\dr\xi\\
=\frac{\la^n}{(2\pi)^n} \sum_{\eta\in \Gamma}\hat{w}(\eps\la
(x+\eta))\,\hat{\chi}_{B^n}(\la(x+\eta))\,.
\end{multline*}
In what follows we set  $\eps=(\la T)^{-1}$. Applying the asymptotic
formula \eqref{asympchar} for $\hat{\chi}_{B^n}$,  one gets
\begin{multline*}
\frac{\la^n}{(2\pi)^n}\,\hat{\chi}_{B^n}(\la(x+\eta)) =\
\frac{2\,\la^{(n-1)/2}\,\sin \left(\la\,|x+\eta| -
\frac{(n-1)\pi}{4}\right)}{(2\pi)^{(n+1)/2}\,|x+\eta|^{(n+1)/2}}\ +\
O\left(\frac{\la^{(n-3)/2}}{|x+\eta|^{(n+3)/2}} \right).
\end{multline*}
Using the formulae above, let us show that $\,\tilde N_{x,0}(\la)\in
B^2\,$ and
\begin{equation}
\label{form61}
 \tilde N_{x,0}(\la)\ \sim\ \sum_{\eta \in \Gamma}
\frac{2\,\sin\left(\la\,|x+\eta|-\frac{(n-1)\pi}{4}\right)}
{(2\pi)^{(n+1)/2}\,|x+\eta|^{(n+1)/2}}\,.
\end{equation}
The proof of \eqref{form61} follows closely \cite[section 3]{Bl2}
and is split into four lemmas.

\begin{lemma}
\label{lem611}
$$
\int_1^T\left| \sum_{\xi\in \Gamma^{*}}e^{2 \pi i \langle
x,\xi\rangle} \left(\chi_{B^n}(2 \pi \xi/\la)-w_{(\la
T)^{-1}}*\chi_{B^n}( 2 \pi \xi/\la)\right) \right|^2
\frac{\dr\la}{\la^{n-1}}\ =\ O(1).
$$
\end{lemma}

\begin{proof}
Set
$$
f_{\xi}(\la) =\chi_{B^n}(2 \pi \xi/\la)-w_{(\la T)^{-1}}*\chi_{B^n}(
2 \pi \xi/\la) = \chi_{\la B^n}(2\pi\xi) - w_{T^{-1}}*\chi_{\la
B^n}( 2 \pi \xi),
$$
where $\la B^n$ is a ball of radius $\la$. Since by definition
$\,\int w_{T^{-1}}(\xi)d\xi = 1\,$ and $w_{T^{-1}} \in
C_0^{\infty}$, it follows that $f_\xi(\la)$ is supported on an
interval of size $O(T^{-1})$ centred at $\la = 2\pi |\xi|$. By a
result of \cite{Colin}, we have
$$
\# \{ \xi \in \Gamma^{*}, |\xi| < R  \}\ =\ \frac{C_n
R^n}{\Vol(\Gamma)} + O(R^{n-2+\frac{2}{n+1}})\,,\qquad\forall
n=1,2,\ldots
$$
Therefore, counting lattice points in an annulus of exterior radius
$|\xi|+O(T^{-1})$ and interior radius $|\xi| -O(T^{-1})$ we get
\begin{equation}
\label{crossterms} \# \{ \beta \in \Gamma^{*}, \supp f_{\beta}(\la)
\cap \supp f_{\xi}(\la) \neq \varnothing \} = O(T^{-1} |\xi|^{n-1} +
|\xi|^{n-2+\frac{2}{n+1}} ).
\end{equation}
This gives us an estimate on the number of non-vanishing cross-terms
in the expression under the integral in Lemma \ref{lem611}. Since
all the terms in this expression are bounded by a constant, taking
into account \eqref{crossterms} and the size of the support of
$f_\xi(\la)$, we obtain
\begin{multline*}
\int_1^T\left| \sum_{\xi\in \Gamma^{*}}e^{2 \pi i \langle
x,\xi\rangle} \left(\chi_{B^n}(2 \pi \xi/\la)-w_{(\la
T)^{-1}}*\chi_{B^n}( 2 \pi \xi/\la)\right) \right|^2
\frac{\dr\la}{\la^{n-1}}\\
\leq\ C T^{-1} \sum_{2 \pi |\xi| \leq T} (T^{-1} +
|\xi|^{-3+2/(n+1)})\ =\ O(1)\,.
\end{multline*}
Note that it is essential for the final equality  that $n < 3$.
\end{proof}

\begin{lemma}
\label{lem612}
\begin{equation*}
\int_1^T\left|\sum_{\eta\in
\Gamma}\frac{|\hat{w}(T^{-1}\,(x+\eta))|}
{\la\,|x+\eta|^{(n+3)/2}}\right|^2 \dr\la\ =\ O(1)\,.
\end{equation*}
\end{lemma}
\begin{proof}
Indeed,
\begin{multline*}
\int_1^T\left|\sum_{\eta\in
\Gamma}\frac{|\hat{w}(T^{-1}\,(x+\eta))|}
{\la\,|x+\eta|^{(n+3)/2}}\right|^2 \dr\la\ \leq \left(\sum_{\eta\in
\Gamma}\frac{|\hat{w}(T^{-1}(x+\eta))|}
{|x+\eta|^{(n+3)/2}}\right)^2\\ \leq \left(\sum_{\eta\in
\Gamma}\frac{1} {|x+\eta|^{(n+3)/2}}\right)^2  =\ O(1).
\end{multline*}
Here the first inequality holds since $\int_{1}^{T} \la^{-2} d\la
\leq 1$, the second inequality is true because $|\hat{w}(x)|=|\int
e^{-i \langle x,\xi\rangle} w(\xi)\,\dr\xi| \leq \int w(\xi)\,\dr\xi
= 1 \,$, and in the last step we use that  $n < 3$.
\end{proof}

\smallskip

In the lemmas below  all the summations are also taken over elements
$\eta \in \Gamma$.

\begin{lemma}
\label{lem613}
$$
\lim_{Q \rightarrow +\infty} \lim_{T \rightarrow +\infty}
\frac{1}{T} \int_1^T\left|\sum_{|\eta| >
Q}\frac{\hat{w}(T^{-1}(x+\eta))} {|x+\eta|^{(n+1)/2}}\;\sin
\left(\la\,|x+\eta|-\frac{(n-1)\pi}{4} \right)\right|^2\dr\la\ = 0
$$
\end{lemma}
The proof of this lemma is a modification of the proof of
\cite[Lemma 3.3]{Bl2}.
\begin{proof}
One can check that
\begin{multline*}
\left| \int_1^T \sin \left(\la\,|x+\eta|-\frac{(n-1)\pi}{4} \right)
\sin \left(\la\,|x+\xi|-\frac{(n-1)\pi}{4} \right) \dr\la\ \right|\\
\leq\ C \min \{ T ,\left||x+\eta|- |x+\xi| \right|^{-1} \}\,.
\end{multline*}
Expanding the squared sum in the integral below and taking into
account that $|\hat{w}(x)| < C(1+|x|)^{-2n}$ as $\hat w$ is rapidly
decreasing, we obtain
\begin{multline}\label{eq612}
\frac{1}{T}\int_1^T\left|\sum_{|\eta| >
Q}\frac{\hat{w}(T^{-1}(x+\eta))} {|x+\eta|^{\frac{n+1}{2}}}\;\sin
\left(\la\,|x+\eta|-\frac{(n-1)\pi}{4} \right)\right|^2\dr\la\\
\leq\ C  \sum_{|\eta| > Q}  \frac{T^{-1} |\eta|^{n-1} +
|\eta|^{n-2+\frac{2}{n+1}}}{|\eta|^{(n+1)}(1+T^{-1}|\eta|)^{4n}}\\
 +\; \frac{C}{T} \sum_{|\eta|>Q} \sum_{k=\lceil
T^{-1}|\eta|^{1-\frac{2}{n+1}}
\rceil}^{\lfloor|\eta|^{1-\frac{2}{n+1}}\rfloor}
\frac{|\eta|^{1-\frac{2}{n+1}}}{k}
\frac{|\eta|^{n-2+\frac{2}{n+1}}}{|\eta|^{(n+1)}(1+T^{-1}|\eta|)^{4n}}\\
+\;\frac{C}{T} \sum_{k=1}^{+\infty} \frac{1}{k} \sum_{j>Q}
\frac{j^{\frac{n-3}{2}}{(j+k)^{\frac{n-3}{2}}}}{(1+T^{-1}j)^{2n}(1+T^{-1}(j+k))^{2n}}\,.
\end{multline}
In the right hand side of \eqref{eq612}, the first sum bounds the
contribution of the terms corresponding to pairs $\,\eta,\xi\,$ such
that $\,||x+\eta|-|x+\xi|| \leq T^{-1}\,$. The second sum estimates
the contribution of the terms such that $\,T^{-1}<
||x+\eta|-|x+\xi||\leq 1$. Here we consider each subinterval of
length $|x+\eta|^{-1+\frac{2}{n+1}}$ separately. The last sum takes
care of pairs $\,\eta,\xi\,$ such that $ k < ||x+\eta|-|x+\xi|| \leq
k+1$, with $k \geq 1$.

For $T>Q$, the first sum in \eqref{eq612} is bounded by
$$
\sum_{|\eta| > Q}  \frac{T^{-1} |\eta|^{n-1} +
|\eta|^{n-2+\frac{2}{n+1}}}{|\eta|^{n+1}(1+T^{-1}|\eta|)^{4n}}\
\leq\
\begin{cases}
C\,(Q^{-1}+T^{-1}Q^{-1}) \,,& n=1  \,,\\
C\,(Q^{-1/3}+ T^{-1}|\ln T - \ln Q|) \,,& n = 2\,.\\
\end{cases}
$$
In order to estimate the second sum in \eqref{eq612}, we first note
that
$$
\sum_{k=\lceil T^{-1}|\eta|^{1-\frac{2}{n+1}}
\rceil}^{\lfloor|\eta|^{1-\frac{2}{n+1}}\rfloor}
\frac{|\eta|^{1-\frac{2}{n+1}}}{k}\ <\ C |\eta|^{1-\frac{2}{n+1}}\,
\ln T.
$$
for all sufficiently large $T$. Estimating  the sum over $\eta$ by
an integral we get
\begin{multline*}
\frac{1}{T} \sum_{|\eta|>Q} \sum_{k=\lceil
T^{-1}|\eta|^{1-\frac{2}{n+1}}
\rceil}^{\lfloor|\eta|^{1-\frac{2}{n+1}}\rfloor}
\frac{|\eta|^{1-\frac{2}{n+1}}}{k}
\frac{|\eta|^{n-2+\frac{2}{n+1}}}{|\eta|^{n+1}\,(1+T^{-1}|\eta|)^{4n}}\\
<\ \frac{C \ln T}{T} \int_{Q}^{+\infty} \frac{ r^{n-3} \dr
r}{(1+T^{-1}r)^{4n}}\ \leq\
\begin{cases}
C_Q\,T^{-1}\ln T \,,& n=1  \,,\\
C_Q\,T^{-1}\ln^2T \,,& n = 2\,.\\
\end{cases}
\end{multline*}
To find a bound for the third sum in \eqref{eq612}, we use the
inequality $2a^\frac{n-3}{2} b^\frac{n-3}{2} \leq a^{n-3}+b^{n-3}$
and once again estimate the sum by an integral,
\begin{multline*}
\sum_{j>Q}
\frac{j^{\frac{n-3}{2}}\,(j+k)^{\frac{n-3}{2}}}{(1+T^{-1}j)^{2n}(1+T^{-1}(j+k))^{2n}}\\
< \ \frac{C}{(1+T^{-1}k)^{2n}} \int_{Q}^{+\infty} \frac{r^{n-3}\dr
r}{(1+T^{-1}r)^{2n}} + C \int_{Q+k}^{+\infty} \frac{r^{n-3}\dr
r}{(1+T^{-1}r)^{2n}}\,.
\end{multline*}
Hence we have
$$
\frac{C}{T}\,\sum_{k=1}^{+\infty} \frac{1}{k} \sum_{j>Q}
\frac{j^{\frac{n-3}{2}}\,(j+k)^{\frac{n-3}{2}}}{(1+T^{-1}j)^{2n}(1+T^{-1}(j+k))^{2n}}
\ \leq\
\begin{cases}
C_Q\,T^{-1} \ln T \,,& n=1  \,,\\
C_Q\,T^{-1} \ln^2 T \,,& n = 2\,.\\
\end{cases}
$$
This completes the proof of the lemma.
\end{proof}

Note that the proof of
Lemma~\ref{lem613} uses the fact that the limit with respect to
$T$ is taken first.

\begin{lemma}
\label{lem614}
$$
\int_1^T\left|\sum_{|\eta| \leq Q}\frac{1-\hat{w}(T^{-1}
(x+\eta))}{|x+\eta|^{(n+1)/2}}\;\sin
\left(\la\,|x+\eta|-\frac{(n-1)\pi}{4} \right)\right|^2\dr\la\ =\
o(T)\,.
$$
\end{lemma}
\begin{proof}
This estimate holds because the sum involves a finite number of
terms and $\,\lim_{T \rightarrow +\infty}\hat{w}(T^{-1} x)=1\,$ for
all $x$.
\end{proof}
By Lemma \ref{lem611} we get
$$
\lim_{T \rightarrow +\infty} \frac{1}{T} \int_1^T\left| N_{x,0}(\la)
- N_{x,0}^{((\la T)^{-1})}(\la) \right|^2 \frac{\dr\la}{\la^{n-1}} =
0\,.
$$
Lemmas \ref{lem612}, \ref{lem613} and \ref{lem614} imply
$$
\lim_{Q \rightarrow +\infty} \lim_{T \rightarrow +\infty}
\frac{1}{T} \int_1^T\left| \la^{\frac{1-n}{2}}N_{x,0}^{((\la
T)^{-1})}(\la) - \sum_{|\eta| \leq
Q}\frac{2\,\sin\left(\la\,|x+\eta|-\frac{(n-1)\pi}{4}\right)}
{(2\pi)^{(n+1)/2}\,|x+\eta|^{(n+1)/2}}\,\right|^2\dr\la\ =\ 0\,,
$$
so that
$$
\lim_{Q \rightarrow +\infty} \lim_{T \rightarrow +\infty}
\frac{1}{T} \int_1^T\left| \tilde N_{x,0}(\la) - \sum_{|\eta| \leq
Q}\frac{2\,\sin\left(\la\,|x+\eta|-\frac{(n-1)\pi}{4}\right)}
{(2\pi)^{(n+1)/2}\,|x+\eta|^{(n+1)/2}}\,\right|^2\dr\la\ =\ 0\,.
$$
This implies \eqref{form61} and therefore proves Examples
\ref{example:circle} and \ref{besic:torus}.
\begin{remark}
\label{difficulty} As was indicated in Remark \ref{remark:tori},
this approach can not work for $n \geq 3$. Indeed, in higher
dimensions we can no longer bound the sums under the integral in
Lemmas \ref{lem611}, \ref{lem612} and \ref{lem613} by taking the
absolute value of each term.  A more delicate analysis is required
in this case (see \cite{BB} for some related results).

We also note that there is a simpler and more direct proof of
Example \ref{example:circle} based on the identity \cite[formula
1.422(4)]{GR}
\begin{equation}
\label{ident:circle} \frac{1}{\sin^2(\frac{s}{2})} =
\sum_{k=-\infty}^{+\infty} \frac{4}{(s+2\pi k)^2}.
\end{equation}
It works in dimension $n=1$ only, and we leave the details to the
interested reader.
\end{remark}

\subsection{Spheres}\label{examples:spheres}
In this subsection we are going to justify
Examples~\ref{intro:spheres} and \ref{besic:sphere}. The spectral
function on $\mathbb{S}^{n}$ depends only on the distance or,
equivalently, on the angle $s$ between the points $x$ and $y$. For
any $x \in \mathbb S^n$, the eigenfunctions orthogonal to the
Legendre polynomials $P_m(n,t)$, with $t = \cos s$ and $m \geq 0$,
vanish at $x$. Hence, $P_m(n,t)$ are the only eigenfunctions that
matter to compute the spectral function. They satisfy the
differential equation
$$
((1-t^2) \partial_t^2 - nt\partial_t) P_m(n,t) = - m(m+n-1)
P_m(n,t)\,,
$$
which is equivalent to the eigenvalue problem $\Delta f = m(m+n-1)
f$ on $\mathbb S^n$ with $f$ depending only on the parameter $t \in
[-1,1]$. Note also that
$$
\int_{-1}^{1}P_m^2(n,t)(1-t^2)^{\frac{n-2}{2}} dt =
\frac{D_n}{D_{n-1}} \frac{1}{N(n,m)}, \quad P_m(n,1) = 1,
$$
where
$$
D_n = \frac{2 \pi^{\frac{n+1}{2}}}{\Gamma(\frac{n+1}{2})}\quad\text
{and}\quad N(n,m) = \frac{(2m+n-1)
\Gamma(m+n-1)}{\Gamma(m+1)\Gamma(n)}
$$
(see \cite[Lemma 10]{Muller}). We assume that $x$ and $y$ are
non-conjugate points, so that $0 < s < \pi$. The rescaled spectral
function on $\mathbb{S}^n$, for $n \geq 2$, is given by
$$
\tilde N_{s}(\mu) = \mu^{\frac{1-n}{2}} N_{s}(\mu)
=\mu^{\frac{1-n}{2}} \sum_{0 \leq m(m+n-1) \leq \mu^2}
\frac{N(n,m)P_m(n,\cos s)}{D_n}
$$
Here we are abusing notation slightly, since according to
\eqref{specfundef} the inequalities in the summation limits should
be strict. However,  because of the rescaling this makes no
difference in $B^2$.

Consider the following generating function (\cite{Muller}, Lemma
17):
\begin{equation}\label{starr0}
\sum_ {m = 0}^{+\infty} N(n,m) P_m(n,t) z^m\ =\
\frac{1-z^2}{(1+z^2-2zt)^{\frac{n+1}{2}}}\,.
\end{equation}
It is easy to see that $\sum_{k=0}^{+\infty} a_k z^k = f(z)$ implies
$\,\sum_{k=0}^{+\infty} \left(\sum_{j=0}^{k} a_j \right) z^k =
\frac{f(z)}{1-z}\,$. Therefore \eqref{starr0} implies
\begin{equation}\label{starr}
\sum_ {m = 0}^{+\infty} N_{s}(\sqrt{m(m+n-1)}) z^m\ =\
\frac{1+z}{D_n (1+z^2-2zt)^{\frac{n+1}{2}}}\,.
\end{equation}

\begin{lemma}
\label{lem621} For $m \in \mathbb{N}$, the spectral function on
$\mathbb{S}^n$ satisfies the following asymptotic formula
\begin{multline*}
N_{s}(\sqrt{m(m+n-1)})\\
=\  \frac{2\cos(\frac{s}{2})\,m^{\frac{n-1}{2}}}{(2 \pi \sin
s)^{\frac{n+1}{2}}}\;\cos\left(\left(\frac{n}{2}+m
\right)s-\frac{(n+1)}{4}\pi\right) + O(m^{\frac{n-3}{2}})\,,\qquad
m\to \infty\,.
\end{multline*}
\end{lemma}

\begin{proof}
Following the approach of  \cite[Chapter VII, Sections  6.6 and
6.7]{Courant}, we prove Lemma \ref{lem621} using the Darboux method
applied to the generating function \eqref{starr}. The Darboux method
relies on the following fact: the coefficients $a_k$ of the Taylor
expansion at the origin of the function $f(z) = \sum_{k=0}^{+\infty}
a_k z^k$, holomorphic in the open disc $\mathbb D$, decay as
$O(k^{-r})$ if $f(e^{ix}) \in C^r(\mathbb R)$. The first step to
obtain the asymptotic formula is to approximate the generating
function \eqref{starr}, taking into account the singularities of the
highest order.

Near the singular point $x=e^{\pm is}$, we can write
\begin{multline*}
\frac{1}{(1+z^2-2zt)^{\frac{n+1}{2}}}\  =\ \frac{\left((e^{\pm
is}-e^{\mp is})+(z-e^{\pm is})\right)^{-\frac{n+1}{2}}}{(z-e^{\pm
is})^{\frac{n+1}{2}}}\\
=\ \frac{e^{\pm i\frac{3(n+1)\pi}{4}}}{(2\sin
s)^{\frac{n+1}{2}}(z-e^{\pm is})^{\frac{n+1}{2}}} \sum_{k =
0}^{+\infty} \binom{-\frac{n+1}{2}}{k}\left( \frac{z-e^{\pm is}}{\pm
2i\sin s} \right)^k,
\end{multline*}
where
$$
\binom{\alpha}{k} =
\begin{cases}
1,& k = 0,\\
\frac{\alpha(\alpha-1)...(\alpha-k+1)}{k!},& k > 0\,.
\end{cases}
$$
For $p > \frac{n+1}{2}$, we have:
\begin{multline}
\label{ququ} \frac{1}{(2\sin s)^{\frac{n+1}{2}}} \sum_{k =
0}^{p}\binom{-\frac{n+1}{2}}{k} \left(
e^{i\frac{3(n+1)\pi}{4}}\frac{(z-e^{is})^{k-\frac{n+1}{2}}}{(2i\sin
s)^k} +
e^{-i\frac{3(n+1)\pi}{4}}\frac{(z-e^{-is})^{k-\frac{n+1}{2}}}{(-2i\sin
s)^k}
\right)\\
- \frac{1}{(1+z^2-2zt)^{\frac{n+1}{2}}}\ \in\ C^1(\bar{\mathbb
D})\,.
\end{multline}
Expanding the powers of $z-e^{\pm is}$ by Taylor's formula at $z=0$,
we see that Taylor's expansion at $z=0$ of the difference
\eqref{ququ} is given by
$$
\frac{2}{(2\sin s)^{\frac{n+1}{2}}} \sum_{k = 0}^{p} \sum_{l =
0}^{+\infty}\binom{-\frac{n+1}{2}}{k} \binom{k-\frac{n+1}{2}}{l}
\frac{\cos
\left((\frac{n+1}{2}+l-k)s+(l-\frac{n+1}{4}-\frac{k}{2})\pi
\right)}{(2\sin s)^k} z^l.
$$
Using \eqref{starr} and comparing the coefficients in front of the
same powers of $z$, we deduce that
\begin{multline*}
D_n N_{s}(\sqrt{m(m+n-1)})\\
= \frac{2}{(2\sin s)^{\frac{n+1}{2}}} \sum_{k = 0}^{p}
\binom{-\frac{n+1}{2}}{k} \left( \frac{\cos
\left((\frac{n+1}{2}+m-k)s+(m-\frac{n+1}{4}-\frac{k}{2})\pi
\right)}{(2\sin s)^k}\binom{k-\frac{n+1}{2}}{m} \right)\\
+ \frac{2}{(2\sin s)^{\frac{n+1}{2}}} \sum_{k = 0}^{p}
\binom{-\frac{n+1}{2}}{k} \left( \frac{\cos
\left((\frac{n-1}{2}+m-k)s+(m-\frac{n+5}{4}-\frac{k}{2})\pi
\right)}{(2\sin s)^k}\binom{k-\frac{n+1}{2}}{m-1} \right)\\
+ O(m^{-1}).
\end{multline*}
 as $m \rightarrow \infty$. Here the error estimate follows from \eqref{ququ}
in view of the remark in the beginning of the proof. Taking into
account the estimate
$$
\frac{\Gamma(m+\frac{n+1}{2})}{\Gamma(m+1)}\ =\ m^{\frac{n-1}{2}} +
O(m^{\frac{n-3}{2}})\,,
$$
we see that the main term of the expansion
above corresponds to $k=0$. Hence we obtain
\begin{multline*}
D_n N_{s}(\sqrt{m(m+n-1)})\\
= \frac{4 \cos(\frac{s}{2})\, \Gamma(m+\frac{n+1}{2})}{(2\sin
s)^{\frac{n+1}{2}}\,\Gamma(m+1)\,\Gamma(\frac{n+1}{2})}\;
\cos\left(\left(\frac{n}{2}+m \right)s-\frac{(n+1)}{4}\pi \right) +
O(m^{\frac{n-3}{2}})\,.
\end{multline*}
This completes the proof of the lemma.
\end{proof}

Note that on $\mathbb S^n$ there are simple expressions for
$N_s(\mu)$ at conjugate points:
$$
N_{0}(\sqrt{m(m+n-1)}) = \frac{(n+2m)\, (n+m-1)!}{D_n\, m!\, n!} =
\frac{2 \pi^{n/2}}{(2\pi)^n \Gamma(n/2) n}\;m^n+O(m^{n-1})\,,
$$
$$
N_{\pi}(\sqrt{m(m+n-1)}) = \frac{(-1)^m (n+m-1)!}{D_n\, m! \,(n-1)!}
\neq o(m^{n-1})
$$
These formulae and Lemma \ref{lem621} prove Example
\ref{intro:spheres}.

\smallskip

Let us now complete the proof of Example \ref{besic:sphere}. Lemma
\ref{lem621} implies that
\begin{multline*}
\int_1^{T} \left| \tilde N_{s}(\la)
-\frac{2\cos(\frac{s}{2})}{(2 \pi \sin s)^{\frac{n+1}{2}}} \cos
\left( \left( \lfloor \la - \frac{n-1}{2} \rfloor +\frac{n}{2}
\right)s-\frac{(n+1)}{4}\pi \right) \right|^2 d \la\\
\leq\ C \sum_{m=1}^T \frac{1}{m} = O(\ln T)\,.
\end{multline*}
Here we
have used that the interval
$$
\left[\sqrt{m(m+n-1)},
\sqrt{(m+1)(m+n)}\right)\,\bigcap\,\{\la\,:\,\lfloor \la -
(n-1)/2\rfloor \ne m\},
$$
where the formula above does not agree with Lemma \ref{lem621}, is
of size $O(m^{-1})$. This follows from the simple asymptotic formula
$\sqrt{m(m+n-1)} = m + \frac{n-1}{2} + O(m^{-1}).$

Therefore, $\tilde N_{s}(\la)$ is $B^2$-equivalent to
$$
\frac{2\cos(\frac{s}{2})}{(2 \pi \sin s)^{\frac{n+1}{2}}}
\cos\left(\left( \lfloor \la - \frac{n-1}{2} \rfloor
+\frac{n}{2}\right)s-\frac{(n+1)}{4}\pi\right)
$$
which can be  rewritten as
\begin{equation}
\label{form62} \frac{1}{(2
\pi)^{\frac{n+1}{2}}\sin^{\frac{n-1}{2}}(s)} \frac{1}{
\sin(\frac{s}{2})} \sin\left(\left( \lfloor \la - \frac{n-1}{2}
\rfloor +\frac{n}{2}\right)s-\frac{(n-1)}{4}\pi\right)
\end{equation}

\begin{lemma}
\label{lem622} The following expansions hold in $B^2$ for $0 < s <
\pi$:
\begin{align*}
\frac{\sin((\lfloor \la \rfloor +
\frac{1}{2})s)}{2\sin(\frac{s}{2})}\ &\sim\ \sum_{k
=-\infty}^{+\infty} \frac{1}{|s + 2\pi k|}\, \sin(\la |s + 2\pi
k|)\,,\\
\frac{\sin(\lfloor \la  + \frac{1}{2} \rfloor
s)}{2\sin(\frac{s}{2})}\ &\sim\ \sum_{k = -\infty}^{+\infty}
\frac{(-1)^k}{|s + 2\pi k|}\, \sin(\la |s + 2\pi k|)\,,\\
\frac{\cos((\lfloor \la \rfloor +
\frac{1}{2})s)}{2\sin(\frac{s}{2})}\ &\sim\ \sum_{k =
-\infty}^{+\infty} \frac{(-1)^{H(k)}}{|s + 2\pi k|}\, \cos(\la |s +
2\pi k|)\,,\\
\frac{\cos(\lfloor \la  + \frac{1}{2} \rfloor
s)}{2\sin(\frac{s}{2})}\ &\sim\ \sum_{k = -\infty}^{+\infty}
\frac{(-1)^{k+H(k)}}{|s + 2\pi k|}\, \cos(\la |s + 2\pi k|)\,.
\end{align*}
Here $H(x)$ is the ``reversed'' Heaviside function: $H(x) = 0$ if $x
\geq 0$, and $H(x) = 1$ if $x < 0$.
\end{lemma}
\begin{proof}
Each of these expansions can be obtained in a similar way as
$\eqref{form61}$ in the case $n=1$. Alternatively, they could be
proved directly using the identity \eqref{ident:circle}.
\end{proof}
Applying Lemma \ref{lem622} to the formula \eqref{form62}, we deduce
that
$$
\tilde N_{s}(\la) \sim \frac{1}{\pi\,(2 \pi\sin s)^{\frac{n-1}{2}}}
\sum_{j = 0}^{\infty} \frac{\sin \left(\la\,l(\gamma_j) -
\frac{(n-1)\pi}{4} - \omega(\gamma_j)\frac{\pi}{2}
\right)}{l(\gamma_j)}\,,
$$
where the sum runs over all geodesic segments $\gamma_j$ starting at
$x$ and ending at $y$. Here $l(\gamma_j)$ is the length of
$\gamma_j$ and $\omega(\gamma_j)$ is the Morse index of $\gamma_j$.
Counting the points conjugate to $x$ on  the geodesic $\gamma$ of
length $|s + 2\pi k|$,  one gets $\omega(\gamma)
=(n-1)\left(2|k|-H(k)\right)$. This completes the proof of Example
\ref{besic:sphere}.

\section{The $\zeta$-function: proofs}\label{zeta}

\subsection{Auxiliary functions}\label{zeta:auxiliary}
Throughout the section we denote
\begin{enumerate}
\item[$\bullet$] $\,t:=\re z\,$, $\,s:=\im z\,$ and
\item[$\bullet$]
$\langle s\rangle:=(1+|s|^2)^{1/2}$.
\end{enumerate}

Let $\rho$ be a real-valued even function satisfying
Condition~\ref{rho} with a sufficiently small $\,\de\,$. Let us fix
an arbitrary $c>0\,$ such that $\,2\,c<\la_1\,$ and
$\,c<d_{x,y}^{-1}\,$ for all $\,x,y\in M\,$, and denote
\begin{equation}\label{z0,z1}
Z_{x,y;j}(z):=z\int_c^\infty
\la^{-z-1}\,N_{x,y;j}(\la)\,\dr\la\,,\qquad j=0,1,
\end{equation}
where $\,N_{x,y;j}(\la)\,$ are the functions defined by
\eqref{n0,n1}. Then
$$
Z_{x,y}\ =\ Z_{x,y;0}(z)\,+\,Z_{x,y;1}(z)
$$
and, due to the finite speed of propagation,
$\left|Z_{x,y}(z)-Z_{x,y;1}(z)\right|\leq C_t\,$ whenever
$\,\de<d_{x,y}\,$.

By Lemma \ref{n1:conv}, we have
\begin{equation}\label{zeta:z1-0}
Z_{x,y;1}(z)\ :=\ \int h_1(z,\mu)\,\dr N_{x,y}(\mu)\,,
\end{equation}
where
$\,h_1(z,\mu):=z\int_c^\infty\la^{-z-1}\,\rho_1(\la-\mu)\,\dr\la\,$.
Substituting
$\,z\la^{-z-1}=-\frac{\dr\hfill}{\dr\la}\left(\la^{-z}\right)\,$,
integrating by parts and then changing variables, we obtain
\begin{equation}\label{zeta:h1-1}
h_1(z,\mu)\ =\
\mu^{-z}+c^{-z}\,\rho_1(c-\mu)-\int_{c-\mu}^\infty(\la+\mu)^{-z}\,
\rho(\la)\,\dr\la\,,\quad\forall\mu>c\,.
\end{equation}
Since one can differentiate under the integral sign, $h_1(z,\mu)$ is
an entire function of the variable $\,z\in\C\,$ smoothly depending
on the parameter $\,\mu\in(2c,+\infty)\,$, such that
\begin{equation}\label{zeta:h1}
\frac{\partial^m}{\partial\mu^m}\,h_1(z,\mu)
=(-1)^m\,\frac{\Gamma(z+m)}{\Gamma(z)}\,h_1(z+m,\mu)\,.
\end{equation}

\begin{lemma}\label{zeta:lemma-h1}
For all $\mu\in(2c,+\infty)$, we have $\,|h_1(z,\mu)|\leq
C_t\,\mu^{-t}\,$ and
\begin{equation}\label{zeta:h1-2}
|h_1(z,\mu)|\ \leq\ C_{t,\vk}\,\mu^{-t-\vk}\,\langle
s\rangle^\vk\,(1+|\ln\langle s\rangle-\ln\mu|)^{-r}\,,
\qquad\forall\vk,r>0\,.
\end{equation}
\end{lemma}

\begin{proof}
The identities \eqref{zeta:h1-1} and $\int\rho(\la)\,\dr\la=1\,$
imply that
\begin{multline*}
h_1(z,\mu)= c^{-z}\rho_1(c-\mu)+\mu^{-z}\int_{-\infty}^{c-\mu}\rho(\la)\,\dr\la
+\int_{c-\mu}^{-\frac{\mu}2}\left(\mu^{-z}-(\mu+\la)^{-z}\right)\rho(\la)\,\dr\la\\
+\int_{\frac{\mu}2}^\infty\left(\mu^{-z}-(\mu+\la)^{-z}\right)\rho(\la)\,\dr\la
+\int_{-\frac{\mu}2}^{\frac{\mu}2}\left(\mu^{-z}-(\mu+\la)^{-z}\right)\rho(\la)\,\dr\la
\,.
\end{multline*}
Since $\rho$ is rapidly decreasing, the first four terms in the
right hand side are bounded by $C_{t,m}\,\mu^{-m}\,$ for all $m>0$.
Therefore we only need to estimate
$\,\int_{-\frac{\mu}2}^{\frac{\mu}2}\left((\mu+\la)^{-z}-\mu^{-z}\right)
\rho(\la)\,\dr\la\,$.

By Taylor's formula,
\begin{multline}\label{zeta:taylor}
(\mu+\la)^{-z}-\mu^{-z} =\sum_{k=1}^{m-1}\,(-1)^k\,
\frac{\Gamma(z+k)}{k!\,\Gamma(z)}\,\la^k\mu^{-z-k}\\
+(-1)^m\,\frac{\Gamma(z+m)}{(m-1)!\,\Gamma(z)}\,\la^m\int_0^1(1-t)^{m-1}\,(\mu+t\la)^{-z-m}\,\dr
t\,,\qquad m=1,2,\ldots
\end{multline}
(unless $-z$ is a nonnegative integer, in which case the terms with
negative exponents $-z-k$ are absent). Substituting
\eqref{zeta:taylor} into the integral, we see that
\begin{multline}\label{zeta:h1-3}
\int_{-\frac{\mu}2}^{\frac{\mu}2}\left((\mu+\la)^{-z}-\mu^{-z}\right)\rho(\la)\,\dr\la\\
=\
\sum_{k=1}^{m-1}\,\frac{(-1)^k\,\Gamma(z+k)}{k!\,\Gamma(z)}\,\mu^{-z-k}
\int_{-\frac{\mu}2}^{\frac{\mu}2}\la^k\,\rho(\la)\,\dr\la\\
+(-1)^m\,\frac{\Gamma(z+m)}{(m-1)!\,\Gamma(z)}\int_{-\frac{\mu}2}^{\frac{\mu}2}
\int_0^1(1-\tau)^{m-1}\,(\mu+\tau\la)^{-z-m}\,\la^m\,\rho(\la)\,\dr\tau\,\dr\la\,.
\end{multline}
Since $\rho$ is rapidly decreasing, Condition~\ref{rho} implies that
the sum in the right hand side of \eqref{zeta:h1-3} is estimated by
$C_m\,\mu^{-t-m}\,\langle s\rangle^{m-1}$ for all $m \geq 1$. On
the other hand, for all $\,j,m\geq0\,$,
\begin{equation}\label{zeta:h1-4}
\int_{-\frac{\mu}2}^{\frac{\mu}2}|(\mu+\tau\la)^{-z-m}\,\la^j\,\rho(\la)|\,\dr\la\
\leq\ C_{t,m}\,\mu^{-t-m}\,,\qquad\forall\tau\in[0,1]\,.
\end{equation}
Therefore the last term in \eqref{zeta:h1-3} does not exceed
$C_{t,m}\,\mu^{-t-m}\,\langle s\rangle^m$. Thus we have
\begin{equation}\label{zeta:h6}
\left|\int_{-\frac{\mu}2}^{\frac{\mu}2}
\left(\mu^{-z}-(\mu+\la)^{-z}\right)\rho(\la)\,\dr\la\,\right|\
\leq\ C_{t,m}\,\mu^{-t}\,b_{s,\mu}^m\,,\qquad\forall m=0,1,\ldots,
\end{equation}
where $b_{s,\mu}:=\mu^{-1}\langle s\rangle$ (the estimate with $m=0$
is a particular case of \eqref{zeta:h1-4} ). Since
$\,b_{s,\mu}^\vk\,(1+|\ln b_{s,\mu}|)^{-r}\geq
C_{r,\vk}\,\min\{b_{s,\mu}^{3\vk/2},b_{s,\mu}^{\vk/2}\}\,$ whenever
$\vk>0$ and
$$
\mu^{-\vk}\,\langle s\rangle^\vk\,(1+|\ln\langle
s\rangle-\ln\mu|)^{-r}=b_{s,\mu}^\vk\,(1+|\ln b_{s,\mu}|)^{-r}\,,
$$
interpolating between the estimates \eqref{zeta:h6}, we arrive at
\eqref{zeta:h1-2}.
\end{proof}

\subsection{Estimates for $Z_{x,y;1}(z)$}\label{zeta:z1}

Lemma~\ref{zeta:lemma-h1} implies the following corollaries.

\begin{corollary}\label{z1:cor1}
$Z_{x,y;1}(z)$ is an entire function of $z$ for each $(x,y)\in
M\times M$, such that
\begin{equation}\label{z1:1}
|Z_{x,y;1}(t+is)|\ \leq\ C_t\left(\langle s\rangle^{n-t}+1\right)\,,
\qquad\forall t\ne n\,.
\end{equation}
\end{corollary}

\begin{proof}
In view of \eqref{zeta:h1} and \eqref{zeta:h1-2}, one can
differentiate under the integral sign in the definition of
$Z_{x,y;1}(z)$. Therefore the function $Z_{x,y;1}(z)$ is analytic.

Let $\tilde
N_{x,y}(\la;a_1,a_2):=\sum_{\la_j<\la}|a_1\,\vf_j(x)+a_2\,\vf_j(y))|^2$
where $a_1,a_2$ are complex constants such that $|a_1|^2+|a_2|^2=2$.
The function $N_{x,y}(\la;a_1,a_2)$ is nondecreasing and, in view of
the Weyl formulae \eqref{intro:bound1}, is estimated by
$\,C\,\la^n$. We have
\begin{align*}
2\,\re N_{x,y}(\la)&=N_{x,y}(\la;1,1)-N_{x,x}(\la)-N_{y,y}(\la)\,,\\
2i\,\im N_{x,y}(\la)&=N_{x,y}(\la;i,1)-N_{x,x}(\la)-N_{y,y}(\la)\,.
\end{align*}
Therefore it is sufficient to prove that the estimate
\eqref{z1:1} holds for the integral $\int_{2c}^\infty
h_1(z,\la)\,\dr G(\la)$, where $G(\la)$ is a nondecreasing
function bounded by $\,C\,\la^n\,$.

If $t<n$ then \eqref{zeta:h1-2} with $\vk=n-t$ and $r=2$ implies
that
\begin{multline*}
\int_{2c}^\infty h_1(z,\la)\,\dr G(\la)\ \leq\ C_t\,\langle
s\rangle^{n-t}\int_{2c}^\infty
\la^{-n}\,(1+|\ln\langle s\rangle-\ln\la|)^{-2}\,\dr G(\la)\\
\leq\ C_t\,\langle s\rangle^{n-t}\int_{2c}^\infty
\la^{-n-1}\,(1+|\ln\langle s\rangle-\ln\la|)^{-2}\,G(\la)\,\dr\la\\
\leq\ C_t\,\langle s\rangle^{n-t}\int_{2c}^\infty
(1+|\ln\langle s\rangle-\ln\la|)^{-2}\,\la^{-1}\,\dr\la\ \leq\
C_t\,\langle s\rangle^{n-t}\,.
\end{multline*}
If $t>n$ then the required estimate is obtained in a similar
way, with the use of the inequality $|h_1(z,\la)|\leq
C_t\,\la^{-t}\,$.
\end{proof}

\begin{remark}\label{z1:remark1}
Note that \eqref{z1:1} holds for $x=y$.
\end{remark}

\begin{corollary}\label{z1:cor2}
\begin{equation}\label{z1:2}
\int_M|Z_{x,y;1}(t+is)|^2\,\dr y\ \leq\
C_t\left(\langle s\rangle^{n-2t}+1\right)\,,\qquad\forall t\ne\frac{n}2\,,
\quad\forall x\in M\,.
\end{equation}
\end{corollary}

\begin{proof}
In the same way as in \eqref{n1:proof1}, we obtain
$$
\int_M|Z_{x,y;1}(z)|^2\,\dr y\ =\ \int|h_1(z,\la)|^2\,\dr
N_{x,x}(\la)\,.
$$
If $\,t<n/2\,$ then \eqref{zeta:h1-2} with $\,\vk=\frac{n}2-t\,$ and
$\,r=2\,$ implies that
\begin{multline*}
\int|h_1(z,\la)|^2\,\dr N_{x,x}(\la)\ \leq\ C_{t,\vk}\,\langle
s\rangle^{n-2t}\int_\eps^\infty\la^{-n}\,(1+|\ln\langle
s\rangle-\ln\la|)^{-2}\,\dr
N_{x,x}(\la)\\
\leq\ C_t\,\langle
s\rangle^{n-2t}\int_\eps^\infty\la^{-n-1}\,(1+|\ln\langle
s\rangle-\ln\la|)^{-2}\, N_{x,x}(\la)\,\dr\la\,.
\end{multline*}
Using the Weyl formula and estimating integrals as in the proof of
the previous corollary, we see that the right hand side is not
greater than $C_t\,\langle s\rangle^{n-2t}$\,.

If $t>\frac{n}2$ then the required inequality is obtained in the
same way, with the use of the estimate $|h_1(z,\la)|^2\leq
C_t\,\la^{-2t}$.
\end{proof}

\subsection{Estimates for $Z_{x,y;0}(z)$}\label{zeta:z0}

If $\,x\ne y\,$ then $\,Z_{x,y;0}(z)\,$ is an entire function of
$\,z\,$. Our goal is to estimate $\,Z_{x,y;0}(z)\,$ uniformly with
respect to $\,s\,$ and $\,x,y\in M\,$. Further on
\begin{enumerate}
\item[$\bullet$]
$O_t(d_{x,y}^{\,p})$ denotes a function which is estimated by
$\,C_t(d_{x,y}^{\,p}+1)\,$.
\end{enumerate}

Obviously, if $\,|f(\la)|\leq C\,\langle\la\rangle^{-m}\,$ for all
$\,\la\in\R_+\,$ then $\,\int_c^\infty\la^{-z}\,f(\la)\,\dr\la\,$ is
an analytic function in the half-plane $\,\{z:t>1-m\}\,$ and
$\,\left|\int_c^\infty\la^{-z}\,f(\la)\,\dr\la\right|\leq C_t\,$ for
all $\,t>1-m\,$. In particular, if $\,\Psi_{x,y}(\la)=N_{x,y}(\la)-
N_{x,y}(-\la)\,$ then, in view of \eqref{n0-Psi}, we have
$$
Z_{x,y;0}(z)=z\int_c^\infty
\la^{-z-1}\,N_{x,y;0}(\la)\,\dr\la=\int_c^\infty
\la^{-z}\left(\rho*\Psi_{x,y}\right)'(\la)\,\dr\la\;+\; O_t(1)\,.
$$

Let us assume that $\de$ in Condition~\ref{rho} is small enough, so
that the Hadamard representation \eqref{n0:hadamard} holds on
$\supp\hat\rho$. Then the Fourier transform of the derivative
$\,\Psi'_{x,y}\,$ admits an asymptotic expansion of the form
\eqref{n0:hadamard} on $\,\supp\hat\rho\,$. This implies that
$$
\left(\rho*\Psi_{x,y}\right)'(\la)\,-\,
\int_c^\infty\rho(\la-\mu)\,\Psi'_{0;x,y}(\mu)\,\dr\mu\ =\
O(\la^{-\infty})\,,
$$
uniformly with respect to $\,x,y\in M\,$, where
$\,\Psi'_{0;x,y}(\mu)\,$ is a function whose asymptotic expansion
for $\,\mu\to\infty\,$ is obtained from \eqref{asymp} by
differentiating each term with respect to $\la\,$ and putting
$\,\la=\mu\,$. Thus we have
\begin{equation}\label{Z0-Psi'0}
Z_{x,y;0}(z)\ =\ \int_c^\infty\int_c^\infty
\la^{-z}\rho(\la-\mu)\,\Psi'_{0;x,y}(\mu)\,\dr\mu\,\dr\la\ +\
O_t(1)\,.
\end{equation}

For all $\,m=0,1,2\ldots\,$, we have
$\,\frac{\dr^m}{\dr\tau^m}\,J_p(\tau)\leq C\,\tau^{p-m}\,$ for small
values of $\,\tau\,$ and
$$
\frac{\dr^m}{\dr\tau^m}J_p(\tau)\ \sim\ \tau^{-\frac12}\, \cos\tau
\sum_{k=0}^\infty a_{m,k}\,\tau^{-k} +\tau^{-\frac12}\,\sin\tau
\sum_{k=0}^\infty b_{m,k}\,\tau^{-k}\,,\quad\tau\to\infty\,,
$$
where $a_{m,k}$ and $b_{m,k}$ are some real coefficients  (see, for
example, \cite[8.440 and 8.451(1)]{GR}). Therefore the asymptotic
formula for $\,\Psi'_{0;x,y}\,$ implies that
$\,|\Psi'_{0;x,y}(\mu)|\leq C\,\langle\mu\rangle^{n-1}\,$ for all
$\,\mu\in(c,d_{x,y}^{\,-1})\,$ and
\begin{multline}\label{Psi'0}
\Psi'_{0;x,y}(\mu)\ =\ \sum_{j+k<N}
u_{j,k}(x,y)\,d_{x,y}^{\,2j+1-n}\,
\left(d_{x,y}\,\mu\right)^{\frac{n-1}2-j-k}\,
\exp\left(i\,d_{x,y}\,\mu\right)\\
+\sum_{j+k<N}v_{j,k}(x,y)\,d_{x,y}^{\,2j+1-n}\,
\left(d_{x,y}\,\mu\right)^{\frac{n-1}2-j-k}\,
\exp\left(-i\,d_{x,y}\,\mu\right)\; +\;R_N(x,y,\mu)
\end{multline}
for all $\,\mu\in(d_{x,y}^{\,-1},\infty)\,$ and $\,N=1,2,\ldots\,$,
where $\,u_{j,k}\,,\,\,v_{j,k}\,$ are bounded functions and
$$
|R_N(x,y,\mu)|\ \leq\ C\,
d_{x,y}^{\,1-n}\,\left(d_{x,y}\,\mu\right)^{\frac{n-1}2-N}\,,
\qquad\forall\mu\in(d_{x,y}^{\,-1},\infty)\,.
$$
In particular, $\,|\Psi'_{0;x,y}(\mu)|\leq
C\,d_{x,y}^{\,1-n}\left(d_{x,y}\,\mu\right)^{\frac{n-1}2}\leq
C\,\mu^{n-1}\,$ whenever $\,\mu\geq d_{x,y}^{\,-1}\,$ and,
consequently, $\,|\Psi'_{0;x,y}(\mu)|\leq
C\,\langle\mu\rangle^{n-1}\,$ for all $\,\mu\geq c\,$.

The above estimates and the obvious inequalities
\begin{equation}\label{la-mu}
\langle\la\rangle^{-1}\langle\mu\rangle\leq2\langle\la-\mu\rangle\,,\qquad
\langle\mu\rangle^{-1}\langle\la\rangle\leq2\langle\la-\mu\rangle
\end{equation}
imply that
\begin{align*}
\left|\la^{-z}\rho(\la-\mu)\,\Psi'_{0;x,y}(\mu)\right|\ &\leq\
C_t\,\langle\la\rangle^{n-t-1}\langle\la-\mu\rangle^{n-1}\,|\rho(\la-\mu)|\,,\\
\left|\la^{-z}\rho(\la-\mu)\,\Psi'_{0;x,y}(\mu)\right|\ &\leq\
C_t\,\langle\mu\rangle^{n-t-1}\langle\la-\mu\rangle^{|t|}\,|\rho(\la-\mu)|
\end{align*}
and
\begin{multline*}
\left|\la^{-z}\rho(\la-\mu)\, R_N(x,y,\mu)\right|\ \leq\
C\,d_{x,y}^{\,t-n+1}\left|\left(d_{x,y}\la\right)^{-t}\rho(\la-\mu)
\,\left(d_{x,y}\mu\right)^{\frac{n-1}2-N}\right|\\
\leq\ C_t\,d_{x,y}^{\,t-n+1} \langle
d_{x,y}\la\rangle^{\frac{n-1}2-N-t}\langle
d_{x,y}\la-d_{x,y}\mu\rangle^{\frac{n-1}2+N}|\rho(\la-\mu)|
\end{multline*}
for all $\la,\mu>d_{x,y}^{\,-1}$. Since the function $\,\rho\,$ is
rapidly decreasing, integrating the above estimates over $\la$ and
$\mu$, we obtain
\begin{equation}\label{Z0-1}
\left|\,\int_c^{d_{x,y}^{\,-1}}\int_c^\infty\la^{-z}\rho(\la-\mu)\,
\Psi'_{0;x,y}(\mu)\,\dr\mu\,\dr\la\,\right|\ \leq\ \begin{cases}
C_t\,(d_{x,y}^{\,t-n}+1)\,,&t\ne n\,,\\
C\,|\ln d_{x,y}|\,,&t=n\,,
\end{cases}
\end{equation}
\begin{equation}\label{Z0-2}
\left|\,\int_{d_{x,y}^{\,-1}}^\infty\int_c^{d_{x,y}^{\,-1}}
\la^{-z}\rho(\la-\mu)\, \Psi'_{0;x,y}(\mu)\,\dr\mu\,\dr\la\,\right|\
\leq\ \begin{cases}
C_t\,(d_{x,y}^{\,t-n}+1)\,,&t\ne n\,,\\
C\,|\ln d_{x,y}|\,,&t=n\,,
\end{cases}
\end{equation}
and
\begin{multline}\label{Z0-3}
\left|\int_{d_{x,y}^{\,-1}}^\infty\int_{d_{x,y}^{\,-1}}^\infty\la^{-z}\rho(\la-\mu)\,
R_N(x,y,\mu)\,\dr\mu\,\dr\la\,\right|\\
\leq\ C_t\,d_{x,y}^{\,t-n-1}\int_1^\infty\int_1^\infty
\langle\la\rangle^{\frac{n-1}2-N-t}\langle
\la-\mu\rangle^{\frac{n-1}2+N}|\rho\left(d_{x,y}^{\,-1}(\la-\mu)\right)|\,\dr\mu\,\dr\la\\
\leq\
C_t\,d_{x,y}^{\,t-n-1}\left(\int_1^\infty\langle\la\rangle^{\frac{n-1}2-N-t}\,\dr\la\right)
\left(\int \langle
\tau\rangle^{\frac{n-1}2+N}|\rho\left(d_{x,y}^{\,-1}\tau\right)|\,\dr\tau\right)
\ =\ O_t(d_{x,y}^{\,t-n})
\end{multline}
for all $\,N>\frac{n+1}2-t\,$.

It remains to estimate the integrals
\begin{multline}\label{Z0-4}
\int_{d_{x,y}^{\,-1}}^\infty\int_{d_{x,y}^{\,-1}}^\infty\la^{-z}\rho(\la-\mu)\,
d_{x,y}^{\,2j+1-n}\, \left(d_{x,y}\,\mu\right)^{\frac{n-1}2-j-k}\,
e^{\pm i\,d_{x,y}\,\mu}\,\dr\mu\,\dr\la\\
=\ d_{x,y}^{\,2j+t-n-1+is}\,\int_1^\infty\int_1^\infty
\la^{-z}\rho\left(d_{x,y}^{\,-1}(\la-\mu)\right)\,\mu^{\frac{n-1}2-j-k}\,
e^{\pm i\,\mu}\,\dr\mu\,\dr\la
\end{multline}
generated by the main asymptotic terms in \eqref{Psi'0}. Expanding
$\,\mu^{\frac{n-1}2-j-k}\,$ by Taylor's formula at $\mu=\la$, 
we see that the right hand side of
\eqref{Z0-4} coincides with
$$
d_{x,y}^{\,2j+t-n-1+is}\int_1^\infty\psi(z,\la,\mu)\left(\int_1^\infty
\rho\left(d_{x,y}^{\,-1}(\la-\mu)\right)\, e^{\pm
i\,\mu}\,\dr\mu\right)\dr\la\;+\;O_t(d_{x,y}^{\,2j+t-n})\,,
$$
where
$$
\psi(z,\la,\mu)\ =\ \sum_{m=0}^l\,\binom{\frac{n-1}2-j-k}{m}
\la^{\frac{n-1}2-j-k-m-z}\,(\mu-\la)^m
$$
with an arbitrary $\,l\geq\frac{n+1}2-t-j-k\,$, and the remainder estimate follows from the inequalities
$$
|\psi(z,\la,\mu) - \la^{-z}\mu^{\frac{n-1}2-j-k}| \leq C \la^{-t} (\mu-\la)^{l+1}
(\mu^{\frac{n-1}2-j-k-l-1} + \la^{\frac{n-1}2-j-k-l-1})
$$
and
$$
\mu^{\frac{n-1}2-j-k-l-1} \leq C \langle\mu-\la\rangle^{\frac{n-1}2+j+k+l+1} \la^{\frac{n-1}2-j-k-l-1}.
$$
Note that the last inequality is a consequence of \eqref{la-mu}.

Obviously,
\begin{multline*}
d_{x,y}^{\,-1}\int_1^\infty\int_{-\infty}^{1}\la^p\,(\mu-\la)^m
\left|\rho\left(d_{x,y}^{\,-1}(\la-\mu)\right)\right|\dr\mu\,\dr\la\\
\leq\ d_{x,y}^{\,m}\int_1^\infty
\int_{d_M^{-1}(\la-1)}^{\infty}\la^p
\left|\mu^m\rho\left(\mu\right)\right|\dr\mu\,\dr\la\ \leq\
C_{p,m}\,,\qquad\forall p\in\R\,,\ \forall m=0,1,2,\ldots,
\end{multline*}
where $\,d_M\,$ is the diameter of $\,M\,$. This estimate and
\eqref{la-mu} imply that the integral \eqref{Z0-4} is equal to
\begin{multline}\label{Z0-psi}
d_{x,y}^{\,2j+t-n-1+is}\,\int_1^\infty\psi(z,\la,\mu)\left(\int_{-\infty}^\infty
\rho\left(d_{x,y}^{\,-1}(\la-\mu)\right)\, e^{\pm
i\,\mu}\,\dr\mu\right)\dr\la\\
=\
\sum_{m=0}^l\,\binom{\frac{n-1}2-j-k}{m}\,d_{x,y}^{\,2j+t-n+m+is}(-i)^m\,
\hat\rho^{(m)}(d_{x,y})\,\int_1^\infty\la^{\frac{n-1}2-j-k-m-z}\,
e^{\pm i\la} \,\dr\la
\end{multline}
modulo $\,O_t(d_{x,y}^{\,2j+t-n})\,$, where $\hat\rho^{(m)}$ is the
$\,m$th derivative of the Fourier transform.

For each $\,p\in\R\,$, the integral $\,\int_1^\infty\la^{p-z}\,
e^{\pm i\la}\,\dr\la\,$ defines an analytic function on the
half-plane $\,\{z:t>p+1\}\,$, where it is bounded on each vertical
line by a constant $\,C_t\,$. This function admits an analytic
continuation to the whole complex plane, obtained by replacing
$\,e^{\pm i\la}\,$ with $\,(\mp i)^m\,\frac{\dr^m}{\dr\la^m}\,e^{\pm
i\la}\,$ and integrating by parts. This continuation coincides with
the difference between the meromorphic continuations of the
integrals $\,\int_0^\infty\la^{p-z}\,e^{\pm i\la}\,\dr\la\,$ and
$\,\int_0^1\la^{p-z}\,e^{\pm i\la}\,\dr\la\,$. According to
\cite[3.381]{GR} and \cite[8.328(1)]{GR},
$$
\int_0^\infty\la^{p-z}\,e^{\pm i\la}\,\dr\la\ =\ \pm\, i\,e^{\pm\,i
\pi(p-z)/2}\,\Gamma(p+1-z)
$$
and
$$
\left|e^{\pm\,i
\pi(p-z)/2}\,\Gamma(p+1-z)\right|=\left|e^{\pi|s|/2}\,\Gamma(p+1-t+is)\right|\leq
C_{p-t}\,|s|^{p-t+1/2}
$$
whenever $\,|s|>1\,$. Replacing $\,\la^{p-z}\,$ with
$\,(m+p-z)^{-1}\ldots(1+p-z)^{-1}\frac{\dr^m}{\dr\la^m}\,\la^{m+p-z}\,$
and integrating by parts, we see that
$\,\left|\int_0^1\la^{p-z}\,e^{\pm i\la}\,\dr\la\,\right|\leq
C_{p-t}\,$ whenever $\,|s|>1\,$. Therefore
$$
\left|\int_1^\infty\la^{p-z}\, e^{\pm i\la}\,\dr\la\,\right|\ \leq\
C_{p-t}\,(1+\langle s\rangle^{p-t+1/2})\,,\qquad\forall p\in\R\,,\ \forall
z\in\C\,.
$$

The above inequality implies that the integral in the left hand side
of \eqref{Z0-psi} is estimated by
$\,C_t\,d_{x,y}^{\,2j+t-n}\,(1+\langle
s\rangle^{\frac{n}2-t-j-k})\,$ and, consequently,
\begin{multline}\label{Z0-5}
\left|\int_{d_{x,y}^{\,-1}}^\infty\int_{\mu>d_{x,y}^{\,-1}}\la^{-z}\rho(\la-\mu)\,
d_{x,y}^{\,2j+1-n}\, \left(d_{x,y}\,\mu\right)^{\frac{n-1}2-j-k}\,
e^{\pm
i\,d_{x,y}\,\mu}\,\dr\mu\,\dr\la\,\right|\\
\leq\ C_t\,d_{x,y}^{\,2j+t-n}\,(1+\langle
s\rangle^{\frac{n}2-t-j-k})\,.
\end{multline}
Now, putting together \eqref{Z0-1}--\eqref{Z0-3} and \eqref{Z0-5},
we obtain
\begin{equation}\label{Z0-6}
\left|Z_{x,y;0}(t+is)\right|\ \leq\ \begin{cases}
C_t\left(d_{x,y}^{\,t-n}\,\langle s\rangle^{\frac{n}2-t}+d_{x,y}^{\,t-n}+1\right)\,,& t\ne n\,,\\
C\left(\langle s\rangle^{\frac{n}2-t}+|\ln d_{x,y}|\right)\,,&
t=n\,.
\end{cases}
\end{equation}

\subsection{Proof of Theorems \ref{zeta:theorem1} and
\ref{zeta:theorem2}}

Theorem \ref{zeta:theorem1} is an immediate consequence of
\eqref{z1:1} and \eqref{Z0-6}. Since the function
$\left(d_{x,y}^{\,2t-n-\eps}+1\right)^{-1}$ is bounded and is
estimated by $\,d_{x,y}^{\,n-2t+\eps}\,$ for small values of
$\,d_{x,y}\,$, Theorem~\ref{zeta:theorem2} follows from \eqref{z1:2}
and \eqref{Z0-6}.

\subsection*{Acknowledgments} The authors would like to thank
Y.~Kannai, G. Paternain, B.~Randol, Z.~Rudnick and
M.~Sodin for useful discussions. We are particularly grateful to
D.~Jakobson for suggesting to look at the almost periodic properties
of the spectral function. The research of H.L. was supported by
NSERC doctoral scholarship. I.P. was supported by NSERC and FQRNT.
Y.S. was supported by the EPSRC grant GR/T25552/01. Part of this
paper was written when Y.S. was visiting the Analysis Laboratory at
the CRM, Montr\'eal. Its hospitality is greatly appreciated.

\end{document}